\documentclass[a4paper,12pt,reqno]{amsart}
\usepackage{amssymb}
\usepackage{amsmath}
\usepackage{enumitem}
\usepackage{mathrsfs}
\usepackage[all]{xy}
\usepackage{mathtools}
\usepackage{hyperref}
\usepackage{url}
\usepackage{bbm}

\usepackage[OT2,T1]{fontenc}
\DeclareSymbolFont{cyrletters}{OT2}{wncyr}{m}{n}

\usepackage[margin=1.3in]{geometry}
\numberwithin{equation}{section} \numberwithin{figure}{section}

\DeclareMathOperator{\Pic}{Pic} 
\DeclareMathOperator{\PDiv}{PDiv} \DeclareMathOperator{\Lines}{Lines}
\DeclareMathOperator{\Gal}{Gal}

\DeclareMathOperator{\Hom}{Hom} \DeclareMathOperator{\re}{Re}

\DeclareMathOperator{\Br}{Br} 

\DeclareMathOperator{\inv}{inv}

\DeclareMathOperator{\HH}{H} 
\let\div\relax
\DeclareMathOperator{\div}{div} 

\DeclareSymbolFont{cyrletters}{OT2}{wncyr}{m}{n}
\DeclareMathSymbol{\Sha}{\mathalpha}{cyrletters}{"58}
\DeclareMathSymbol{\Be}{\mathalpha}{cyrletters}{"42}

\newcommand{\OO}{\mathcal{O}}

\newcommand\FF{\mathbb{F}}
\newcommand\PP{\mathbb{P}}
\renewcommand\AA{\mathbb{A}}
\newcommand\ZZ{\mathbb{Z}}
\newcommand\NN{\mathbb{N}}
\newcommand\QQ{\mathbb{Q}}
\newcommand\RR{\mathbb{R}}
\newcommand\CC{\mathbb{C}}
\newcommand\GG{\mathbb{G}}

\newcommand\Gm{\GG_\mathrm{m}}
\newcommand{\Adele}{\mathbf{A}}

\newcommand{\sU}{\mathcal{U}}

\newcommand{\Zp}{\mathbb{Z}_p}

\newcommand{\Qp}{\mathbb{Q}_p}

\newcommand{\bu}{{\bf u}}

\newcommand{\by}{{\bf y}}

\renewcommand{\(}{\left(}
\renewcommand{\)}{\right)}

\newcommand{\nequiv}{\not\equiv}
\DeclareMathOperator{\valp}{v_{\it p}}
\DeclareMathOperator{\valn}{v}

\newcommand{\ve}{\varepsilon}
\newtheorem{lemma}{Lemma}

\newtheorem{theorem}[lemma]{Theorem}
\newtheorem{proposition}[lemma]{Proposition}
\newtheorem{corollary}[lemma]{Corollary}

\theoremstyle{definition}

\newtheorem{remark}[lemma]{Remark}

\numberwithin{lemma}{section}

\usepackage[usenames, dvipsnames]{color}

\begin{document}
\title[Markoff surfaces]{Integral Hasse principle and strong approximation for Markoff surfaces}
\author{Daniel Loughran}
  \address{
  Department of Mathematical Sciences \\
University of Bath \\
Claverton Down \\
Bath \\
BA2 7AY \\
UK}

\author{Vladimir Mitankin}
	\address{Max Planck Institut f\"{u}r Mathematik \\
		Vivatsgasse 7 \\
		53111 Bonn \\
		Germany
	}
	\email{vmitankin@mpim-bonn.mpg.de}

%\classno{11D45 (primary), 14G05, 11N37 (secondary)}
%\extraline{The author is sponsored by an EPRSC student scholarship.}

\begin{abstract}
	We study the failure of the integral Hasse principle and strong approximation for Markoff surfaces, as studied by Ghosh and Sarnak, using the Brauer--Manin obstruction.
\end{abstract}

\maketitle
\thispagestyle{empty}
\tableofcontents

\section{Introduction} \label{sec:intro}
For each $m \in \ZZ$, we consider the affine surfaces
\begin{equation} \label{def:Um}
  U_m: \quad u_1^2 + u_2^2 + u_3^2 - u_1u_2u_3 = m.
\end{equation}
We denote by $\sU_m$ the integral model of $U_m$ defined over $\ZZ$ by the same equation. In a recent paper \cite{GS17}, Ghosh and Sarnak studied integral points and failures of the integral Hasse principle for such surfaces. In our paper we extend their analysis using the Brauer--Manin obstruction. 

Here we say that $\mathcal{U}_m$ \emph{fails the integral Hasse principle} if $\sU_m(\Adele_{\ZZ}) \neq \emptyset$ but $\sU_m(\ZZ) = \emptyset$, where $\Adele_\ZZ = \RR \times \prod_p \ZZ_p$. We say that $U_m$ \emph{satisfies weak approximation} if the image of $U_m(\QQ)$ in $\prod_{v} U_m(\QQ_v)$ is dense, where the product is over all places of $\QQ$. Finally, we say that $\sU_m$ \emph{satisfies strong approximation} if $\sU(\ZZ)$ is dense in  $\sU_m(\Adele_\ZZ)_{\bullet}:=\pi_0(U_m(\RR)) \times \prod_{p} \sU_m(\ZZ_p)$, where $\pi_0(U_m(\RR))$ denotes the set of connected components of $U_m(\RR)$. Note that we work with $\pi_0(U_m(\RR))$ since $\sU(\ZZ)$ is \emph{never} dense in  $U_m(\RR)$ for simple topological reasons.

We first note that the natural compactification of \eqref{def:Um} in $\PP^3$ is a smooth cubic surface for $m \neq 0,4$, and the hyperplane at infinity consists of three coplanar lines. In particular, the rational points on $U_m$ are Zariski dense by \cite{Kol02}, so that $U_m(\QQ) \neq \emptyset$.
On the other hand, by \cite[Prop.~6.1]{GS17} for every integer $m$ we have $\sU_m(\Adele_{\ZZ}) \neq \emptyset$ unless $m \equiv 3 \bmod 4$ or $m \equiv \pm 3 \bmod 9$.  In particular, a positive proportion of these surfaces have an $\Adele_{\ZZ}$-point.

Our first theorem shows that these surfaces almost always fail weak approximation.

\begin{theorem}	 \label{thm:WA}
	We have
	$$\#\{ m \in \ZZ : 
		|m| \leq B,
		\text{ $U_m$ satisfies weak approximation}\} \ll B^{1/2}.$$
\end{theorem}

Theorem \ref{thm:WA} is sharp, since weak approximation holds when $m-4$ is a square as such surfaces are rational (see Lemma \ref{lem:rational}). 

Our method shows that for $\gg B^{1/2}$ of $m$, there are elements of $\prod_{p \mid 2(m-4)} \sU_m(\ZZ_p)$ which cannot be approximated by a rational point, so that the failure of weak approximation yields a failure of strong approximation. More naively: for almost all $m$, there are solutions to the equation \eqref{def:Um} modulo some integer (depending on $m$) which cannot be realised by a rational solution.

In fact, we are also able to show that there is almost always a failure of strong approximation which is \emph{not} explained by a failure of weak approximation. To make this precise: Let $\overline{U_m(\QQ)}$ be the closure of $U_m(\QQ)$ in $\prod_v U_m(\QQ_v)$ and let $\overline{\sU_m(\ZZ)}$ be the closure of $\sU_m(\ZZ)$ in $\sU_m(\Adele_\ZZ)_{\bullet}$.  We clearly have $\overline{U_m(\QQ)} \subseteq \prod_v U_m(\QQ_v)$ and $\overline{\sU_m(\ZZ)} \subseteq \sU_m(\Adele_\ZZ)_{\bullet} \cap \overline{U_m(\QQ)}$ (here we abuse notation slightly and consider the image of $\overline{U_m(\QQ)}$ inside $\pi_0(U_m(\RR)) \times \prod_p U_m(\QQ_p)$).  Theorem \ref{thm:WA} says that the former  inclusion is strict for almost all $U_m$; our next theorem shows that the second inclusion is also strict for almost all $U_m$.

\begin{theorem} \label{thm:SA}
	We have
	$$\#\{ m \in \ZZ : 
		|m| \leq B, \sU_m(\Adele_\ZZ) \neq \emptyset
		\text{ but } \overline{\sU_m(\ZZ)} = \sU_m(\Adele_\ZZ)_{\bullet} \cap \overline{U_m(\QQ)}
	\} \ll B^{1/2}.$$
\end{theorem}

Explicit examples illustrating these results can be found in \S \ref{sec:BM_examples}. On the other hand, we show that if there is a Brauer--Manin obstruction to the integral Hasse principle, then there are only finitely many explicit possibilities for $m-4$ modulo squares. This result may be viewed as an analogue of the finiteness of exceptional spinor
classes in the study of the representation of an integer by a ternary quadratic form, as studied in \cite[\S 7]{CTX09}.

\begin{theorem} \label{thm:235}
	Assume that $m \in \ZZ$ is such that $\sU_m$ has a Brauer--Manin
	obstruction to the integral	Hasse principle. 
	Then
	$$m-4  \bmod \QQ^{*2} \in \langle \pm 1, 2,3,5\rangle \subset \QQ^*/\QQ^{*2}.$$
\end{theorem}

This qualitative statement shows that there is a Brauer--Manin obstruction to the integral Hasse principle for at most $O(B^{1/2})$ of the surfaces $U_m$ for $m \in \ZZ$ with $|m| \le B$. A more in-depth analysis of the Brauer--Manin obstruction allows us to prove that not only is $m-4$ a product of small primes times a square, but all the prime divisors of $m-4$ must satisfy very strong congruence conditions. This allows us to show the improved upper bound $O(B^{1/2}/ (\log B)^{1/2})$, which is sharp by the following theorem.

\begin{theorem} \label{thm:Br asympt}
	We have
	$$\#\{ m \in \ZZ : |m| \leq B, \sU_m(\Adele_\ZZ) \neq \emptyset 
	\text{ but } \sU_m(\Adele_\ZZ)^{\Br} = \emptyset \} \asymp \frac{B^{1/2}}{(\log B)^{1/2}}.$$	
\end{theorem}

Note that \cite[Thm.~1.2.(i)]{GS17} claims to obtain the lower bound $B^{1/2}/(\log B)^{1/4}$ for the number of Hasse failures. In fact we shall see in the proof of Theorem~\ref{thm:Br asympt} that their method only gives the lower bound $B^{1/2}/(\log B)^{1/2}$, which agrees with Theorem \ref{thm:Br asympt}.

General methods were developed in \cite{BBL16} and \cite{Bri18} which, when they apply, show that almost all of the varieties in the family fail weak approximation and almost all also have no Brauer-Manin obstruction to the Hasse principle. These results do not apply here as they concern smooth projective varieties. But a similar method may be employed in our case to obtain bounds of the shape $O(B/(\log B))$ in the setting of Theorems \ref{thm:WA} and \ref{thm:Br asympt} (see \cite[Prop.~6.1]{BBL16}). In particular, our upper bounds are stronger than what these general methods would give.

In \cite{GS17}, Ghosh and Sarnak presented numerical evidence that at least $B^\gamma$, for some $1/2 < \gamma < 1$, of the surfaces $U_m$ fail the integral Hasse principle for $|m| \leq B$. In particular, Theorem \ref{thm:Br asympt} implies that almost all of these hypothetical Hasse failures are not explained by the Brauer--Manin obstruction. In our last result we make modest progress towards this by confirming the existence of infinitely many such surfaces.

\begin{theorem} \label{thm:no_Br}
  We have
  \begin{equation*}
    \#\{ m \in \ZZ : |m| \leq B, \sU_m(\Adele_\ZZ)^{\Br} \neq \emptyset 
    \text{ but } \sU_m(\ZZ) = \emptyset \} 
    \gg \frac{B^{1/2}}{\log B}.
  \end{equation*}
\end{theorem}

We prove our results by explicitly calculating the Brauer group of the surfaces $U_m$ for ``general'' $m$ (namely, for all $m$ outside an explicit thin set). We find that $\Br U_m /\Br \QQ \cong (\ZZ/2\ZZ)^3$, generated by explicit quaternion algebras. Once we know the Brauer group, we then perform a detailed analysis of the Brauer--Manin obstruction associated to these quaternion algebras. The Hasse failures which we use in Theorem \ref{thm:no_Br} were already considered in \cite{GS17}, where reduction theory was used to show the failure of the Hasse principle. We show that there is no Brauer--Manin obstruction in this case using our knowledge of the Brauer group.

An interesting feature of our results is that the Brauer group of $U_m$ over $\bar{\QQ}$ is isomorphic to $\QQ/\ZZ$, with the Galois invariant part being $\ZZ/2\ZZ$. Nonetheless we will show that the non-trivial Galois invariant element does not descend to $\QQ$ for general $m$, hence the transcendental Brauer group is in fact trivial. A similar phenomenon was observed by Colliot-Th\'el\`ene and Wittenberg in their study of sums of three cubes \cite[Prop.~3.1]{CTW12}.

The structure of the paper is as follows. In \S\ref{sec:2} we study affine cubic surfaces given by the complement of three coplanar lines and their transcendental and algebraic Brauer groups. In \S\ref{sec:3} we turn our attention to the natural smooth projective compactifictions of the surfaces \eqref{def:Um}, where we explicitly calculate the algebraic Brauer group of the compactification. In \S\ref{sec:4} we complete our analysis of the Brauer group by calcuating the Brauer group of the affine surface \eqref{def:Um}. We then use the Brauer group in \S\ref{sec:5} to give explicit examples of Brauer-Manin obstructions to the integral Hasse principle and strong approximation, and  prove the results from the introduction.

Whilst preparing our paper for publication, we learnt of the work of Colliot-Th\'el\`ene, Wei and Xu \cite{CTWX}, also studying Markoff surfaces, which  appeared on arXiv very shortly after our work. There is an overlap between our papers concerning the calculations of the Brauer groups  in \S 2,3,4 and some of the examples of the Brauer-Manin obstruction in \S\ref{sec:BM_examples}. However, our methods for calculating the Brauer groups are on the whole quite different, with us opting for more geometric arguments whereas they use more algebraic approach. Our work contains a more detailed analysis of the possibilities for the local invariant maps in \S \ref{sec:invariants}. This is crucial in the proofs of our main theorems on the Brauer--Manin obstruction to the integral Hasse principle (Theorems \ref{thm:235} and \ref{thm:Br asympt}), which do not follow from their work. In \cite{CTWX} they use reduction theory to show that strong approximation \emph{always}
fails, which improves on our Theorems~\ref{thm:WA} and \ref{thm:SA}. They also obtain an improvement over Theorem \ref{thm:no_Br}, by giving a lower bound with $B^{1/2}/(\log B)$ replaced by $B^{1/2}/(\log B)^{1/2}$ \cite[Thm~5.8]{CTWX}. It would be interesting to find a lower bound here of the shape $B^{1/2}/(\log B)^{\gamma}$ for some $\gamma < 1/2$, which, combined with Theorem \ref{thm:Br asympt}, would show that almost all counter-examples the integral Hasse principle are not explained by the Brauer-Manin obstruction. 

\subsubsection*{Notation}
For a variety $X$ over a field $k$, we denote by $\Br X = \HH^2(X,\Gm)$ its Brauer group. We let $\Br_1 X = \ker (\Br X \to \Br X_{\bar{k}})$ denote the \emph{algebraic Brauer group} of $X$. Elements of $\Br X$ which do not belong to $\Br_1 X$ are called \emph{transcendental}.

\subsection*{Acknowledgements} We thank Tim Browning, Jean-Louis Colliot-Th\'el\`ene, Dasheng Wei, Olivier Wittenberg, and Fei Xu for useful discussions, and Martin Bright and Dami\'an Gvirtz for help with \texttt{magma} computations. The first-named author is supported by EPSRC grant EP/R021422/1. The second-named author would like to thank CAPES for the financial support while working on this project.

\section{Geometry of affine cubic surfaces} \label{sec:2}
By an \emph{affine cubic surface}, we mean an affine surface of the form
$$U: f(u_1,u_2,u_3) = 0$$
where $f$ is a polynomial of degree of $3$. The closure of $U$ in $\PP^3$ is a cubic surface $S$. The complement $H=S\setminus U$ is a hyperplane section on $S$. Much of the geometry of $U$ can be understood in terms of the geometry of $S$ and $H$. We begin with some basic remarks.

\subsection{Basic geometry}

\begin{lemma} \label{lem:no_functions}
	Let $S$ be a smooth cubic surface over a field $k$, let $H$ be a hyperplane section
	and set $U = S \setminus H$. Then $\OO(U)^* = k^*$.
\end{lemma}
\begin{proof}
	To prove the result, we may assume that $k$ is algebraically closed.
	Suppose that there is a non-constant function $f \in \OO(U)^*$ and consider 
	$D = \div f$. As $f$ is invertible on $U$, we see that $D$ is supported on the
	boundary $H$. Write $D = \sum_{i \in I} a_i D_i$ as a sum of irreducible divisors,
	for some index set $I$, where $a_i \neq 0$. We shall consider the various possibilities 
	for the $D_i$.
	
	If $\# I = 1$ then we find that either $D$ or $-D$ is a non-zero principal effective divisor; this is clearly
	a contradiction. If $\# I = 2$, we have either
	\begin{enumerate}
		\item $D_1$ and $D_2$ are both lines.
		\item $D_1$ is a line and $D_2$ is an irreducible conic.
	\end{enumerate}
	In the first case we have $D_1^2 = D_2^2 = -1$ and $D_1\cdot D_2 = 1$.
	As $D$ is principal we find that
	$$ 0 = D \cdot D_1 = -a_1 + a_2, \quad 0 = D \cdot D_2 = a_2 - a_1, \quad 0 = D \cdot H = a_1 + a_2.$$
	This is clearly a contradiction.
	In the second case we have $D_1^2 = -1, D_2^2= 0$ and $D_1 \cdot D_2 = 2$. 
	We find that
	$$ 0 = D \cdot D_1 = -a_1 + 2a_2, \quad 0 = D \cdot D_2 = 2a_1,$$	
	which again gives a contradiction.
	
	If $\#I = 3$ then the $D_i$ are all lines. A similar argument to the above also
	gives a contradiction. We deduce that $I = \emptyset$, so that $D=0$ and so 
	$f \in k^*$.
\end{proof}

\begin{lemma} \label{lem:blow-up}
	Let $S$ be a smooth surface over a field $k$ of characteristic $0$.
	Let $\pi:S' \to S$ be the blow-up of $S$ in a rational point $P$ with exceptional curve $E$.
	Then inclusion and pull-back via $\pi$ induces isomorphisms 
	$$\Br (S' \setminus E) \cong \Br S' \cong \Br S \cong \Br (S\setminus P).$$
\end{lemma}
\begin{proof}
	Let $\pi:S' \to S$ be the blow-up map and $E$ the exceptional curve of the
	blown-up point $P$. Then we have
	$$\Br S' \subset \Br (S' \setminus E) \cong \Br (S \setminus P) = \Br S,$$
	by Grothendieck's purity theorem. However pull-back clearly gives $\Br S \subset \Br S'$,
	as required.
\end{proof}

\subsection{Cubic surfaces with $3$ coplanar lines}
\subsubsection{The conic bundle}
Let $S$ be a smooth cubic surface over a field $k$ with a line $L$. 
Let $H$ be a hyperplane containing $L$. Then $S \cap H = L \cup C$,
where $C$ is a (possibly singular) plane conic. Varying $H$ we obtain
a conic bundle structure
$$\pi: S \to \PP^1$$
on $S$. This is not an arbitrary conic bundle: the line $L$ gives a \emph{degree
$2$ multisection} of $\pi$, i.e.~the induced map $\pi_L:L \to \PP^1$ has degree $2$.
Moreover $\pi$ has $5$ singular fibres over $\bar{k}$, each given as the union
of $2$ lines meeting in a single point.

Let $P \in \PP^1_k$ be a closed point such that $\pi^{-1}(P)$ is singular;
this consists of $2$ lines over the algebraic closure of the residue field
$\kappa(P)$ of $P$. We define the \emph{residue} of $\pi$ at $P$ to be
the class in $\kappa(P)^{*2}/\kappa(P)^* = \HH^1(k,\ZZ/2\ZZ)$ corresponding to the splitting
field of these irreducible components.

\subsubsection{Brauer group}
There has been much work already  on the Brauer groups of affine cubic surfaces when the hyperplane section $H$ is \emph{smooth}; here the Brauer group is closely related to the torsion points on the Jacobian of $H$ (see for example \cite{CTW12} or \cite{BL17}).

We shall be interested in the case where the hyperplane $H$ is \emph{singular}. We focus on the case where $H$ is given by $3$ coplanar lines, though similar results also hold when $H$ is a union of a line and smooth plane conic.

To help calculate the Brauer group of the surface, we will require
the following version of the Gysin sequence (see \cite[Cor.~VI.5.3]{Mil80} or
\cite[Thm.3.4.1]{CT95}).

\begin{lemma}[Gysin sequence] \label{lem:Gysin}
	Let $X$ be a smooth variety over a field $k$ of characteristic $0$ and 
	$Z \subset X$ a smooth divisor. Let  $n \in \NN$ and $U=X \setminus Z$. Then there is an exact sequence
	\begin{align*}
	0 &\to \HH^{1}(X,\mu_n) \to \HH^{1}(U,\mu_n) \to \HH^0(Z,\ZZ/n\ZZ) \to\\
	&\to \HH^{2}(X,\mu_n) \to \HH^{2}(U,\mu_n) \to \HH^{1}(Z,\ZZ/n\ZZ) \to \cdots
	\end{align*}
\end{lemma}

Our first result calculates the Brauer group over the algebraic closure as a module for the absolute Galois group $G_k = \Gal(\bar{k}/k)$. In the statement, we let $\mu_\infty = \varinjlim \mu_n$ be the direct limit of all groups $\mu_n$ of $n$th
roots of unity in $\CC$. Moreover, recall that  an \emph{Eckardt point} on a cubic surface is a point where three lines meet.

\begin{proposition} \label{prop:trans_Galois}
	Let $S$ be a smooth cubic surface over a field $k$ of characteristic $0$.
	Let $H \subset S$ be a hyperplane section which is the union of $3$ lines $L_1,L_2,L_3$ and let $U = S \setminus H$.
	Then as $G_k$-modules we have 
	$$\Br U_{\bar{k}} \cong 
	\begin{cases}
		0, & \text{if $L_1,L_2,L_3$ meet in an Eckardt point},\\
		\QQ/\ZZ(-1):=\Hom(\mu_\infty, \QQ/\ZZ), & \text{otherwise}.
	\end{cases}
	$$
	In particular, if $k$ contains no non-trivial roots of unity and $L_1,L_2,L_3$
	do not meet in an Eckardt point then
	$$(\Br U_{\bar{k}})^{G_k} \cong \ZZ/2\ZZ.$$
\end{proposition}
\begin{proof}
	Blow down $L_3$ to obtain a del Pezzo surface $\psi: S \to S'$ of degree $4$
	and let $C_i = \psi(L_i)$. Then by Lemma \ref{lem:blow-up} 
	we have $\Br (S\setminus H) \cong \Br (S' \setminus (C_1 \cup C_2))$.
	The curves $C_1$ and $C_2$ are smooth conics on $S'$. Moreover, we
	have $-K_{S'} = C_1 + C_2$ in $\Pic S'$ and 
	$$ \#(C_1 \cap C_2)=	\begin{cases}
		1, & \text{if $L_1,L_2,L_3$ meet in an Eckardt point},\\
		2, & \text{otherwise}.
	\end{cases}$$
	Let $U_1 = S' \setminus C_1$ and $U_2 = U_1 \setminus C_2$.
	We apply the Gysin sequence from Lemma~\ref{lem:Gysin} to $(S',U_1)$
	to find the
		exact sequence
	\begin{align*}
		\HH^{2}(S'_{\bar{k}},\mu_n) \to \HH^{2}(U_{1,\bar{k}},\mu_n) \to \HH^{1}(C_{1,\bar{k}},\ZZ/n\ZZ).
	\end{align*}
	As $C_1 \cong \PP^1$ we have $\HH^{1}(C_{1,\bar{k}},\ZZ/n\ZZ) = 0$, hence
	the	map $\HH^{2}(S'_{\bar{k}},\mu_n) \to \HH^{2}(U_{1,\bar{k}},\mu_n)$ is surjective. 	Applying Gysin to $(S',U_1)$ again we obtain
	\begin{align*}
	\HH^3(S'_{\bar{k}}, \mu_n) \to \HH^3(U_{1,\bar{k}}, \mu_n) \to \HH^2(C_{1,\bar{k}}, \ZZ/n\ZZ) \to \HH^4(S'_{\bar{k}}, \mu_n)
		\to \HH^4(U_{1,\bar{k}}, \mu_n).
	\end{align*}
	As $U_1$ is non-proper we have $\HH^4(U_{1,\bar{k}}, \mu_n) = 0$. The map 
	$\HH^2(C_{1,\bar{k}}, \ZZ/n\ZZ) \to \HH^4(S'_{\bar{k}}, \mu_n)$ is therefore an isomorphism as both groups have the same cardinality. As 
	$\HH^3(S'_{\bar{k}}, \mu_n) = 0$, it follows that $\HH^3(U_{1,\bar{k}}, \mu_n) = 0$.
	We now apply the Gysin sequence to $(U_1,U_1 \cap C_2)$ to find
	$$\HH^{2}(U_{1,\bar{k}},\mu_n) \to \HH^{2}(U_{2,\bar{k}},\mu_n) 
	\to \HH^{1}(U_{1,\bar{k}} \cap C_{2,\bar{k}},\ZZ/n\ZZ) \to \HH^3(U_{1,\bar{k}}, \ZZ/n\ZZ) = 0.$$
	From the Kummer sequence we have the commutative diagram
	with exact rows
	$$
	\xymatrix{0 \ar[r] & (\Pic S'_{\bar{k}})/n \ar[r] \ar[d]^{f} &	\HH^2(S'_{\bar{k}},\mu_n)  \ar[d] \ar[r]& (\Br S'_{\bar{k}})[n]  \ar[d] \ar[r] & 0   \\
	0 \ar[r] & (\Pic U_{2,\bar{k}})/n \ar[r]&	\HH^2(U_{2,\bar{k}},\mu_n) \ar[r]& (\Br U_{2,\bar{k}})[n] \ar[r] & 0.}
	$$
	As $S'$ is smooth the map $f$ is surjective. Moreover, 
	$\Br S'_{\bar{k}} = 0$ as $S'$ is a smooth projective geometrically
	rational surface.
	As the map $\HH^{2}(S'_{\bar{k}},\mu_n) \to \HH^{2}(U_{1,\bar{k}},\mu_n)$
	is surjective, we therefore deduce the Galois equivariant isomorphism 
	$$(\Br U_{2,\bar{k}})[n] \cong \HH^{1}(U_{1,\bar{k}} \cap C_{2,\bar{k}},\ZZ/n\ZZ).$$
	We may now complete the proof.
	If $L_1,L_2,L_3$ meet in an Eckardt point then $U_1 \cap C_2 \cong \AA^1$, thus
	$\HH^{1}(U_{1,\bar{k}} \cap C_{2,\bar{k}},\ZZ/n\ZZ) = 0$.
	Otherwise we have $U_1 \cap C_2 \cong \Gm$. In this case 
	Kummer theory yields the isomorphism  
	$\HH^{1}(\mathbb{G}_{m,\bar{k}},\mu_n)
	= \ZZ/n\ZZ$, thus twisting  coefficients shows that 
	$\HH^{1}(\mathbb{G}_{m,\bar{k}},\ZZ/n\ZZ) = \Hom(\mu_n,\ZZ/n\ZZ)$.
	The result now follows on applying
	this to all $n$, as the Brauer group is torsion.
\end{proof}

\begin{proposition} \label{prop:list}
	Let $S$ be a smooth cubic surface over a field $k$ of characteristic $0$.
	Let $H \subset S$ be a hyperplane section which is the union of $3$ lines $L_1,L_2,L_3$ and let $U = S \setminus H$. Then $\Pic U_{\bar{k}}$ is torsion
	free and $\Br_1 U / \Br k \cong \HH^1(k, \Pic U_{\bar{k}})$ is isomorphic
	to one of the following groups
	\[ 0, \, \ZZ/4\ZZ, \, \ZZ/2\ZZ \times \ZZ/4\ZZ, \, (\ZZ/2\ZZ)^r \, (r=1,2,3,4). \]
\end{proposition}
\begin{proof}
	The isomorphism $\Br_1 U / \Br k \cong \HH^1(k, \Pic U_{\bar{k}})$ is a well-known
	consequence of the fact that $U(k) \neq \emptyset$, $\OO(U_{\bar{k}})^* = \bar{k}^*$
	(Lemma \ref{lem:no_functions}) and the Hochshild--Serre spectral sequence
	(see e.g.~\cite[Lem.~6.3(iii)]{San81}). 
	We have the exact sequence
	\begin{equation} \label{seq:Pic}
		0 \to \ZZ^3 \overset{i}{\to} \Pic \bar{S} \to \Pic \bar{U} \to 0
	\end{equation}
	of $G_k$-modules, where $i(n_1,n_2,n_3) = n_1[L_1] + n_2[L_2] + n_3[L_3]$.
	Moreover, we have the exact sequence
	$$0 \to \PDiv \bar{S} \to \Lines \bar{S} \to \Pic \bar{S} \to 0,$$
	where $\Lines \bar{S}$ is the free abelian group generated by the $27$
	lines of $\bar{S}$ and $\PDiv \bar{S}$ is the subgroup of principal
	divisors. These two sequences give an explicit description of 
	$\Pic \bar{U}$ as a quotient of a permutation module by a submodule.
	Moreover, the absolute Galois group $\Gal(\bar{k}/k)$
	acts on $\Lines \bar{S}$ via a subgroup of the Weyl group
	$W(E_6)$, which is well-defined	up to conjugacy.
	
	This data can all be fed into \texttt{magma}.
	One enumerates all $350$ conjugacy classes of subgroups of $W(E_6)$
	together with the corresponding action on $\ZZ^{27}$. One finds
	that $48$ of these subgroups correspond to smooth cubic
	surfaces with $3$ coplanar lines. One then constructs $\Pic \bar{U}$
	together with the action of the Galois group $G$
	using the above sequences and computes
	$\HH^1(G, \Pic \bar{U})$ with standard commands in \texttt{magma}.
	(Note that, by inflation-restriction 
	$\HH^1(k, \Pic \bar{U}) =\HH^1(G, \Pic \bar{U})$, as
	one checks in \texttt{magma} that $\Pic \bar{U}$ is
	a free $\ZZ$-module).
	This gives the possible list of Brauer groups stated in the proposition.
\end{proof}

\begin{remark}
It is interesting to see the group $\ZZ/4\ZZ$ occurring here. This occurs for a unique  group action on the lines by a group of order $8$. The  group $\ZZ/2\ZZ \times \ZZ/4\ZZ$ occurs for a unique group action by a group of order $4$.
\end{remark}

\begin{remark} \label{rem:magma}
	There is a direct proof, without \texttt{magma},
	that every element
	of $\Br_1 U / \Br k$ is $4$-torsion. 
	Consider the exact sequence \eqref{seq:Pic} and define the map
	$$j: \Pic \bar{S} \to \ZZ^3, \quad j: [D] \mapsto ([D] \cdot ([L_2] + [L_3]), [D] \cdot ([L_1] + [L_3]), [D] \cdot ([L_1] + [L_2])). $$
	A simple calculation shows that $j\circ i$
	is multiplication by $2$ on $\ZZ^3$.
	We now apply cohomology to obtain the exact sequence
	$$0 \to \Br_1 S / \Br k \to \Br_1 U / \Br k \to \HH^2(k, \ZZ^3) \overset{i^*}{\to}
	\HH^2(k, \Pic \bar{S}),$$
	on using the vanishing $\HH^1(k, \ZZ^3)=0$. As $S$ is a conic bundle surface,
	it is well-known that $\Br_1 S / \Br k$ is $2$-torsion (this follows for
	example from \cite[Thm.~2.11]{Lou}).
	However for $b \in \Br_1 U / \Br k$, its image in $\HH^2(k, \ZZ^3)$
	is easily seen to have order $2$ since $j^*\circ i^*$ is multiplication
	by $2$ on $\HH^2(k, \ZZ^3)$. The result follows.	
\end{remark}

\begin{remark}
	The complement of a \emph{smooth} hyperplane section on a smooth cubic
	surface over an algebraically closed field always has non-trivial Brauer group
	(see e.g.~\cite[\S2.2]{BL17}).
	Proposition \ref{prop:trans_Galois} shows however that this is not the case for 
	the complement of a singular hyperplane section given by three lines meeting
	in an Eckardt point.
\end{remark}

\section{Geometry of projective Markoff surfaces} \label{sec:3}
We now consider the geometry of the natural compactifications of the Markoff
surfaces. We work over a field $k$ of characteristic $0$ and study
the cubic surfaces with the equation
\begin{equation} \label{def:Sm}
	S_m: \quad x_0(x_1^2 + x_2^2 + x_3^2) - x_1x_2x_3 = m x_0^3 \qquad \subset \PP^3,
\end{equation}
for $m \in k$. We assume throughout that $S_m$ is smooth; this is equivalent to $m \neq 0,4$.
The surfaces $S_m$ contain the three coplanar lines
$$L_i: x_0 = 0, x_i = 0.$$
In particular, each $S_m$ comes equipped with $3$ conic bundle structures. There is an obvious
action of the symmetric group of order $3$ on $S_m$ which permutes these lines. We 
focus our attention on the line $L_3$ and denote the associated conic bundle by $\pi:S_m \to \PP^1$.
Analysing the conic bundle structure, one finds that
this is given by
\begin{equation} \label{eqn:conic_bundle}
	S_m: \quad s(x^2 + y^2) - txy + s(t^2 - ms^2)z^2 = 0 \qquad \subset \FF(0,0,1)
\end{equation}
where $\FF(0,0,1) := \PP(\OO_{\PP^1} \oplus \OO_{\PP^1} \oplus\OO_{\PP^1}(1))$
is viewed as the quotient $((\mathbb{A}^2\setminus 0) \times (\mathbb{A}^{3}\setminus 0) )/ \Gm^2$ for the action
$$
	(\lambda,\mu) \cdot (s,t;x,y,z) = (\lambda s,\lambda  t; \mu x, \mu y, \lambda^{-1} \mu z).
$$
(See \cite[\S2]{FLS18} for more information about how to write
down equations for conic bundle surfaces.)
One sees that this surface is isomorphic to the original $S_m$  via the map
\begin{equation} \label{eqn:blow_up}
	\FF(0,0,1) \to \PP^3, \quad (x,y,z;s,t) \mapsto (sz,x,y,tz),
\end{equation}
which also realises $\FF(0,0,1)$ as the blow-up of $\PP^3$ in the line $L_3$.
The conic bundle map $\pi$ is given by mapping to $(s:t)$.

\begin{lemma} \label{lem:conic_bundle}
	The following holds.
	\begin{enumerate}
		\item The map $\pi: S_m \to \PP^1$ has $5$ singular geometric fibres. 
		\item The discriminant is given by
		$$\Delta(s,t) = s(t-2s)(t+2s)(t^2 - ms^2).$$
	\end{enumerate}
	Assume that $m \notin k^{*2}$.
	Let $P_1,P_2,P_3,P_4$ be the points corresponding the zero locus of $\Delta(s,t)$ in $\PP^1$.
	Let $k_i = \kappa(P_i)$ be the residue fields of the $P_i$.
	\begin{enumerate}
		\setcounter{enumi}{2}
		\item The fibre over each $P_i$ has the following residue in $k_i$.
		$$P_1: 1, \quad P_2: m-4, \quad P_3: m-4, \quad P_4: m-4.$$
	\end{enumerate}
\end{lemma}
\begin{proof}
	Follows immediately from the explicit equation \eqref{eqn:conic_bundle}
	and a simple calculation.
\end{proof}

\begin{lemma} \label{lem:Br of comp}
	Assume that $[k(\sqrt{m}, \sqrt{m-4}):k] = 4$. Then
	$$\Br S_m / \Br k \cong (\ZZ/2\ZZ),$$
	with generator given by the quaternion algebra
  $$\alpha = ((x_3/x_0)^2-4, m-4).$$
\end{lemma}
\begin{proof}
	Note that our assumptions imply that $(m - 4) \notin k(\sqrt{m})^{*2}$.
	Lemma \ref{lem:conic_bundle} and
	standard formulae for Brauer groups of conic bundle surfaces
	(e.g.~\cite[Thm.~2.11]{Lou}) show that $\Br S_m / \Br k$
	is generated by $((t/s)^2 - 4, m-4)$. With respect to the map
	\eqref{eqn:blow_up}, this gives the stated quaternion algebra.
\end{proof}

We finish with the following observation.

\begin{lemma} \label{lem:rational}
	The surface $S_m$ is rational if and only if $m-4 \in k(\sqrt{m})^{*2}$.
\end{lemma}
\begin{proof}
	If  $m-4 \notin k(\sqrt{m})^{*2}$, then one verifies in a similar manner
	to the proof of Lemma~\ref{lem:Br of comp} that $\Br S_m / \Br k $
	is non-trivial, hence $S_m$ is non-rational. On the other hand, 
	if $m-4 \in k(\sqrt{m})^{*2}$ and $m \notin k^{*2}$,
	then the fibre over $P_4$ 
	is split (see Lemma \ref{lem:conic_bundle}). A component in 
	the fibre can therefore be blown-down. Blowing down a component over $P_1$,
	it follows that $S_m$ is birational to a conic bundle surface 
	with at most $2$ singular fibres over $\bar{k}$. Such surfaces
	are well-known to be rational once they have a rational point
	(see e.g~\cite[\S 1]{KM17} and the references therein).
	A similar argument applies if $m \in k^{*2}$ and $m -4 \in k^{*2}$,
	and completes the proof.
\end{proof}

\section{Brauer group of affine Markoff surfaces} \label{sec:4}
We now calculate the Brauer groups of the affine surfaces 
\begin{equation*}
  U_m: \quad u_1^2 + u_2^2 + u_3^2 - u_1u_2u_3 = m.
\end{equation*}
It will be convenient to have an alternative shape for this equation, as in \cite[\S 8]{GS17}. Let $i, j, k$ be distinct members of the set $\{1, 2, 3\}$. Then
\begin{equation} \label{eqn:GS_1}
	U_m : \quad (2u_i - u_ju_k)^2 - (m - 4)2^2 = (u_j^2 - 4)(u_k^2 - 4).
\end{equation}
The change of variables $w = 2u_i - u_ju_k$ here corresponds to blowing down the line $L_i$ (as already used in Proposition \ref{prop:trans_Galois}). In particular, this shows the alternative 
formulae  
\begin{equation} \label{eq:alpha different reps}
  \alpha = (u_1^2 - 4, m - 4)   = (u_2^2 - 4, m - 4)   = (u_3^2 - 4, m - 4)
\end{equation}
for the element $\alpha$ from Lemma \ref{lem:Br of comp}.

\subsection{Transcendental Brauer group}

\begin{proposition} \label{prop:trans}
	Let $k$ be a field of characteristic $0$ and let
	$m \in k^*$ be such that $m-4 \in k^{*2}$.
	Then the natural map $\Br S_m \to \Br_1 U_m$ is an isomorphism.
	Moreover, if $k$ contains no non-trivial roots of unity,
	then the natural map $\Br S_m \to \Br U_m$ is an isomorphism.
\end{proposition}
\begin{proof}
	Our argument is inspired by the proof of \cite[Prop.~3.1]{CTW12}.
	Let $b \in \Br U_m$. To show that $b \in \Br S_m$, it suffices
	to show that $b$ is unramified along the three lines $L_i$ on $S_m$.
	To do so, we are free to multiply $b$ by constant algebra,
	so by Proposition \ref{prop:trans_Galois} and Remark \ref{rem:magma}
	we may assume that $b$ has order dividing $4$.
	
	Let $L=L_3$ and $C = L_1 \cup L_2$.
	Let $U_L = L \setminus C$. Note that $L$ meets $C$ in two rational points, hence
	$L$ is non-canonically isomorphic to $\Gm$. We choose the point $(0:1:1:0) \in U_L$
	to be the identity element of the group law. Then the isomorphism with $\Gm$
	is realised via
	\begin{equation} \label{eqn:Gm}
		\Gm \to S_m, \quad t \mapsto (0:1:t:0).
	\end{equation}
	The residue	of $b$ along $L$ lies inside $\HH^1(U_L, \ZZ/4\ZZ)$.
	We will show by contradiction that it is trivial. So assume
	that the residue is non-trivial. Replacing $b$ by $2b$, if necessary,
	we may assume that the residue has order $2$.
	This means that the residue
	corresponds to some irreducible degree $2$ finite \'etale cover $f:V \to U_L$.
	
	As $m-4 \in k^{*2}$, the fibres $C_2$ and $C_3$ over $P_2$ and $P_3$, respectively,
	are both split (i.e.~a union of two lines over $k$). It turns
	out that $L$ meets $C_2$ and $C_3$ each in exactly one point
	with multiplicity two;
	i.e.~these are Eckardt points on the surface. Let $Q_2=(0:1:1:0)$
	and $Q_3=(0:1:-1:0)$ be the corresponding rational points (note that
	$Q_2$ is the identity element of $U_L$).
	Let $F_2$ be an irreducible component of $C_2$ and consider the restriction
	of $b$ to $F_2$. This is well-defined outside of $Q_2$. 
	However $F_2 \setminus Q_2 \cong \mathbb{A}^1$ has constant Brauer group,
	so $b$ in fact extends to all of $F_2$. As $F_2$ meets $L$ transversely,
	we deduce from \cite[Prop.~4.15]{LTBT18}
	that the evaluation of the residue of $b$ at $Q_2$ is also trivial,
	so that $f^{-1}(Q_2)$ consists of exactly $2$ rational points. This shows that
	$V$ is geometrically irreducible, and hence $V \cong \Gm$ non-canonically.
	
	If $b \in \Br_1 U_m$, then the residue lies in $\HH^1(k,\ZZ/2\ZZ)$,
	so if it is non-trivial then it corresponds to some quadratic
	extension of $k$. Hence if it is non-trivial then 
	$V$ must be geometrically irreducible, which contradicts the above.
	We deduce that $b$ is unramified along $L$.
	However, running the same argument with the other $2$ lines shows that $b$
	is unramified hence $b \in \Br S_m$ as claimed. 
	
	Assume now that $k$ contains no non-trivial roots of unity.	
	Choosing a rational point over $Q_2$ and using \eqref{eqn:Gm},
	we may therefore identity the cover $V \to U_L$ with the map
	\begin{equation} \label{eqn:cover}
		\Gm \to S_m, \quad t \mapsto (0:1:t^2:0).
	\end{equation}
	However, we may run the exact same argument with $C_3$ to deduce that 
	the fibre of $f$ over $Q_3 = (0:1:-1:0)$ contains a rational point.
	But our assumptions imply that $-1 \notin k^{*2}$,
	which is a contradiction. Thus the residue of $b$ along $L$ is trivial. 
	Considering the other lines, as above, we conclude that $b$ is everywhere unramified,
	hence $b \in \Br S_m$. 
\end{proof}

In \cite[\S 4]{CTWX}, further cases are given where the transcendental Brauer group is trivial. For completeness we give our own proof of these results using our method.

\begin{proposition} \label{prop:trans_2}
	Let $k$ be a field of characteristic $0$ and let
	$m \in k^*$ be such that $m \in k^{*2}$ and $m-4 \in k^{*2}$.
	If $\frac{\sqrt{m} + \sqrt{m-4}}{2} \notin k^{*2}$,
	then the natural map $\Br S_m\{2\} \to \Br U_m\{2\}$ is an isomorphism. 
\end{proposition}
\begin{proof}
	The proof is a minor variant of the proof of Proposition \ref{prop:trans}.
	Let $b$ have reside of order $2$ along $L_3$. 
	As in the proof of Proposition \ref{prop:trans}, we find that this residue
	corresponds to \eqref{eqn:cover}.
	Since $m$ and $m-4$ are both squares in $k$,
	each fibre of the conic bundle morphism is split. We consider the lines in one
	of the fibres over $P_4$. 
	They are contained in the hyperplane $x_0 = \sqrt{m}x_3^2$ and given by
	the equation $x_1^2 + x_2^2 - \sqrt{m}x_1x_2 = 0$. The components
	are
	$$2x_1 = (\sqrt{m} \pm \sqrt{m-4})x_2.$$
	These meet the line $L_3$ in the points
	$$\left(0:1: \frac{\sqrt{m} \pm \sqrt{m-4}}{2}:0\right).$$
	As in the proof of Proposition \ref{prop:trans}, we find that the fibres
	over these points contain a rational point. But $\frac{\sqrt{m} + \sqrt{m-4}}{2}$
	is not a square by assumption, which is a contradiction so
	$b$ is unramified along $L_3$.
	Arguing with the lines $L_1$ and $L_2$ shows that $b$ is everywhere
	unramified, as required.
\end{proof}

\begin{corollary} \label{cor:trans}
	Let $m \in \QQ^*$ with $m \neq 0,4$. If $-(m-4)$
	is a square in $\QQ$, then assume further that 
	$\frac{\sqrt{m}+ \sqrt{m-4}}{2}$ is not a square in 
	$\QQ(\sqrt{m}, \sqrt{m-4})$.
	Then $U_m$ has trivial transcendental Brauer group. 
\end{corollary}
\begin{proof}
	Let $b \in \Br U_m$ be a transcendental Brauer group element.
	Modifying $b$ by a constant algebra, we may assume that 
	$b$ has order $2$ by Proposition \ref{prop:trans_Galois}.
	Assume first that $-(m-4)$ is not a square. 
	Then $\QQ(\sqrt{m-4})$ contains no non-trivial root of unity of even order,
	hence by (the proof of) Proposition~\ref{prop:trans}
	the element becomes unramified over $\QQ(\sqrt{m-4})$. 
	But $S_m$ has trivial Brauer group
	over $\bar{\QQ}$, which is a contradiction.
	
	On the other hand, if $-(m-4)$
	is a square in $\QQ$, then  by assumption
	$\frac{\sqrt{m}+ \sqrt{m-4}}{2}$ is not a square in 
	$\QQ(\sqrt{m}, \sqrt{m-4})$ .
	The result in this case now follows from a similar
	application of Proposition \ref{prop:trans_2}.
\end{proof}

It is shown in \cite[\S 4]{CTWX} that if the assumptions of Corollary \ref{cor:trans} fail, then the transcendental Brauer group can in fact be non-trivial. In particular, Corollary \ref{cor:trans} is sharp.

\begin{remark}
	Note that $(\Br U_{m,\bar{\QQ}})^{G_\QQ} = \ZZ/2\ZZ$
	by Proposition \ref{prop:trans_Galois}. Nevertheless, in 
	Corollary \ref{cor:trans} the Galois invariant element
	of order $2$ does not descend to a Brauer group element over $\QQ$.
\end{remark}

\subsection{Algebraic Brauer group}

\begin{proposition} \label{prop:Br1 Um}
	Let $k$ be a field of characteristic $0$ and $m \in k^*$ such that 
	$[k(\sqrt{m}, \sqrt{m-4}):k] = 4$.
	Then $\Br_1 U_m/ \Br k \cong (\ZZ/2\ZZ)^3$. A complete set
	of representatives for the non-trivial elements are given by the quaterion
	algebras
	$$\alpha_{i,\pm} = (u_i \pm 2, m-4), \, i \in \{1,2,3\},
	 \quad \alpha = (u_1^2 - 4,m-4),$$
	which satisfy the following relations in $\Br_1 U_m$:
	\begin{align}
		\alpha_{i,-} + \alpha_{i,+} &= \alpha, \quad i \in \{1,2,3\}, \label{eqn:alpha_-_+} \\
		\alpha_{1,+} + \alpha_{2,+} + \alpha_{3,+} &= 0, \label{eqn:sum_alpha_i_+} \\
		\alpha_{1,-} + \alpha_{2,-} + \alpha_{3,-} &= \alpha. \label{eqn:sum_alpha_i}
	\end{align}
	Moreover $\Br_1 U_m/ \Br k $ is generated by the quaternion algebras
	$\alpha_{i,-}$ for $i \in \{1,2,3\}.$
\end{proposition}
\begin{proof}
	We first explain why $\alpha_{i,\pm} \in \Br_1 U_m$.
	It suffices to show that $\alpha_{i,\pm}$ is unramified
	along the divisor $u_i \pm 2 = 0$. However, this is
	one of the non-split singular fibres in the conic bundle,
	thus $m-4$ is a square in the function field by Lemma \ref{lem:conic_bundle} whence 
	$\alpha_{i,\pm}$ is indeed unramified  along this divisor.
	
	We now show that the $\alpha_{i,\pm}$ give distinct elements
	of the Brauer group. One calculates that $\alpha_{i,\pm}$
	is unramified along the line $L_i$, but has residue $m-4$
	along $L_j$ for $j \neq i$, $j \in \{1,2,3\}$, on the compactification
	$S_m$ \eqref{def:Sm}. Thus
	$\alpha_{i,\pm} \neq \alpha_{j,\pm}$ for $i \neq j$.
	Moreover the relation \eqref{eqn:alpha_-_+} is trivially verified,
	which shows that $\alpha_{i,-} \neq \alpha_{i,+}$ in $\Br_1 U_m$.
	Next one uses the equation \eqref{def:Um}
	to deduce that
	$$(u_1 + 2)(u_2 + 2)(u_3 + 2) 
	= (u_1 + u_2 + u_2 + 2)^2 - (m-4),$$
	whence \eqref{eqn:sum_alpha_i_+} easily follows. 
	Then \eqref{eqn:sum_alpha_i} follows from  \eqref{eqn:sum_alpha_i_+}
	and \eqref{eqn:alpha_-_+}.	
	This shows that
	the $\alpha_{i,-}$ generate a subgroup of $\Br_1 U_m$
	isomorphic to $(\ZZ/2\ZZ)^3$.
	
	To show that these also generate $\Br_1 U_m/ \Br k$, one can just construct
	the group action in our list from Proposition \ref{prop:list}
	using \texttt{magma}, and see that 
	$\Br_1 U_m/ \Br k \cong (\ZZ/2\ZZ)^3$ in this case.
	
	For completeness, we also give a geometric argument
	that we have found all the elements.
	Let $\beta \in \Br_1 U_m/\Br k$. By Proposition \ref{prop:trans},
	we find that $\beta$ becomes unramified over $k(\sqrt{m-4})$.
	This implies that the residue of $\beta$ along each line
	$L_i$ is killed after this extension, thus the residue is either
	$1$ or $m-4$. If $\beta$ is unramified on all $L_i$ then
	$\beta \in \{0,\alpha\}$. So assume $\beta$ is ramified along
	some $L_i$.	
	First note that $\beta$ cannot be ramified
	along just one of the $L_i$ by Lemma \ref{lem:blow-up}.
	It also follows that $\beta$ cannot we ramified along all $3$
	lines; indeed then $\beta + \alpha_{1,-}$ would be only ramified
	along $L_1$, which is impossible. Hence $\beta$ is ramified
	along exactly $2$ lines; say $L_2$ and $L_3$. But then $
	\beta + \alpha_{1,-}$ is unramified, hence $\beta + \alpha_{1,-}
	\in \{0,\alpha\}$, whence $\beta =\alpha_{1,-}$ or $\alpha_{1,+}$
	by our above arguments.
	This shows that $\beta$ is in the list of already found
	Brauer group elements, hence we are done.
\end{proof}

\begin{corollary} \label{cor:generic}
	Let $m \in \QQ^*$ be such that none of $m,(m-4)$, nor $m(m-4)$
	is a square in $\QQ$. Moreover, if $-(m-4)$ is a square in $\QQ$, then assume that 
	$\frac{\sqrt{m}+ \sqrt{m-4}}{2}$ is not a square in 
	$\QQ(\sqrt{m}, \sqrt{m-4})$.
	Then $\Br U_m/ \Br \QQ \cong (\ZZ/2\ZZ)^3$, generated by the quaternion
	algebras $\alpha_{i,-}$.
\end{corollary}
\begin{proof}
	Our assumptions imply that $[\QQ(\sqrt{m}, \sqrt{m-4}):\QQ] = 4$.
	The result then follows from Corollary \ref{cor:trans} and
	Proposition \ref{prop:Br1 Um}.
\end{proof}

\section{The Brauer--Manin obstruction} \label{sec:5}
We now calculate the Brauer--Manin obstruction for $U_m$ for $m \in \ZZ$.

\subsection{Set-up} \label{sec:set-up}
Recall that $U_m$ is given by \eqref{def:Um} and $\sU_m$ is its obvious integral model over $\ZZ$. For convenience let $\bu = (u_1, u_2, u_3)$. 

We make great use of the alternative equation \eqref{eqn:GS_1}.
This gives us an alternative expression of the model $\sU_m$, at least away from $2$. 
\begin{equation} \label{eqn:GS}
	\sU_m \otimes \ZZ[1/2] : \quad (2u_i - u_ju_k)^2 - (m - 4)2^2 = (u_j^2 - 4)(u_k^2 - 4).
\end{equation}

\emph{We assume throughout \S\ref{sec:BM}--\ref{sec:invariants} that $m \in \ZZ$ satisfies the hypotheses of Corollary~\ref{cor:generic}}.  Thus the group $\Br U_m / \Br \QQ$ is generated by the quaternion algebras $\alpha_{i, -}$ from Proposition~\ref{prop:Br1 Um}. We let $\mathcal{A} = \langle \alpha_{1,-}, \alpha_{2,-}, \alpha_{3,-} \rangle$ be the subgroup of $\Br U_m$ generated by these elements. We will also make use of the element $\alpha$, its alternative representations \eqref{eq:alpha different reps} and the relations from Proposition \ref{prop:Br1 Um}.

\subsubsection{Assumptions of Corollary \ref{cor:generic}} \label{sec:assumptions}
Let us note that the assumptions of Corollary \ref{cor:generic} are not very restrictive if one is interested in Brauer--Manin obstructions to the integral Hasse principle. Indeed, assume  that these conditions fail. Then
$$m-4,m, m(m-4), \text{ or } -(m-4) \in \QQ^{*2}.$$
In the first case we have $\Br U_m = \Br \QQ$ by Lemma \ref{lem:rational} and  Proposition \ref{prop:trans}. In the second case, there is no obstruction
to the integral Hasse principle, as there is always the integral
point $(\sqrt{m},0,0)$. The third case does not occur under the assumption that $U_m$ is smooth, as the following shows.

\begin{lemma} \label{lem:m(m-4)}
	Let $m \in \ZZ$ be such that $m(m-4)$ is a square.
	Then $m = 0,4$.
\end{lemma}
\begin{proof}
	Consider the Diophantine equation $m(m-4) = a^2$.
	If this has an integral solution, then the discriminant, viewed
	as a polynomial in $m$, must be a square. This shows that
	$16 + 4a^2$ is a square, hence $4 + a^2 = b^2$ for some $b$.
	But then this factorises as $(b+a)(b-a) = 4$, which is easily
	checked to imply that $a = 0$ and $b \in \{\pm 2\}$,
	hence $m(m-4) = 0$, as claimed.
\end{proof}

We  calculate explicitly the conditions in Corollary~\ref{cor:generic} when $-(m - 4) \in \QQ^{* 2}$.

\begin{proposition} \label{prop:no trans in general}
	Let $d$ be a positive integer and  $m = 4 - d^2$.
	Then $\frac{\sqrt{m}+ \sqrt{m-4}}{2}$ is a square in 
	$\QQ(\sqrt{m}, \sqrt{m-4})$ if and only if
	$d = 2(n^2 \pm 1)$ for some positive $n \in \ZZ$.
\end{proposition}

\begin{proof}
  Let $i = \sqrt{-1}$ and let $K = \QQ\(\sqrt{m - 4}, \sqrt{m}\) = \QQ(i, \sqrt{4 - d^2})$. If $d = 2$, then $\frac{\sqrt{m}+ \sqrt{m-4}}{2} = i \notin K^2$. So assume throughout that $d \neq 2$.
 One easily verifies (for example by taking $1, i$ as a basis for the Gaussian integers over $\ZZ$) that $4 - d^2$ is not a square in $\QQ(i)$ as $d \neq 2$, and thus $[K:\QQ] = 4$. Our condition is
  \begin{equation} \label{eqn:trans element conditon}
    \frac{di + \sqrt{4 - d^2}}{2} \in K^{\times 2}.
  \end{equation}
	Observe that $2i = (1 + i)^2$ and thus $2i \in K^{\times 2}$. Multiplying \eqref{eqn:trans element conditon} by $2i$ shows that it is equivalent to $
    -d + i\sqrt{4 - d^2} \in K^{\times 2}.$
  Since $1, i, \sqrt{4 - d^2}, i\sqrt{4 - d^2}$ are a basis for $K$ as a $\QQ$-vector space we deduce that \eqref{eqn:trans element conditon} holds if and only if
  \begin{equation} \label{eqn:trans element condition2}
    -d + i\sqrt{4 - d^2} = \(a_0 + a_1i + a_2\sqrt{4 - d^2} + a_3i\sqrt{4 - d^2}\)^2 \neq 0
  \end{equation}
  for some $a_0, a_1, a_2, a_3 \in \QQ$. Comparing the coefficients of $i$ and $\sqrt{4 - d^2}$ we get
  \begin{equation} \label{eqn:system a_i}
    \begin{split}
      a_0a_1 + a_2a_3(4 - d^2) &= 0, \\
      a_0a_2 - a_1a_3 &= 0.
    \end{split}
  \end{equation}
	Firstly, assume that $a_2a_3 \neq 0$. The bottom equation in \eqref{eqn:system a_i} gives
  \begin{equation*}
    \frac{a_1}{a_2} = \frac{a_0}{a_3} = \mu \in \QQ.
  \end{equation*}
  Dividing through by $a_2a_3$ in the top equation of \eqref{eqn:system a_i} translates it into $\mu^2 = d^2 - 4$. The equation $4 + \mu^2 = d^2$ was analysed in the proof of Lemma~\ref{lem:m(m-4)} and it is insoluble as $d \neq \pm 2$. Hence \eqref{eqn:trans element condition2} is not satisfied.

	Therefore we may assume that $a_2a_3 = 0$. We consider first the case $a_3 = 0$, so that the equations \eqref{eqn:system a_i} become $a_0a_1 = a_0a_2=0$. If $a_1 = a_2 = 0$, then \eqref{eqn:trans element condition2} is not satisfied as $K \neq \QQ$. If $a_0 = 0$, then \eqref{eqn:trans element condition2} becomes
  \begin{equation*}
      -d + i \sqrt{4 - d^2}
      = -a_1^2 + a_2^2(4 - d^2) + 2a_1a_2i\sqrt{4 - d^2}.
  \end{equation*}
  Comparing the terms shows that  $a_2 = 1/2a_1$, and yields the biquadratic equation
  \begin{equation*}
    4a_1^4 - 4d a_1^2 - (4 - d^2) = 0.
  \end{equation*}
  It is soluble over $\QQ$ and its solutions are of the shape $a_1^2 = \pm 1 + d/2$. Since there are no half integers which are squares in $\QQ$ we must have $d$ even, and that $d = 2(n^2 \pm 1)$ for some positive integer $n$.

  The case $a_2 = 0$ is completely analogous to $a_3 = 0$, except one obtains the biquadratic equation $4a_0^4 + 4d a_0^2 - (4 - d^2) = 0$. It has no real solutions as $d \neq 2$ and thus \eqref{eqn:trans element condition2} in soluble for $a_2 = 0$.
  This completes the proof.
\end{proof}

\subsection{Review of the Brauer--Manin obstruction} \label{sec:BM}
We briefly recall how the Brauer--Manin obstruction works in our setting, following \cite[\S8.2]{Poo17} and \cite[\S1]{CTX09}.
For each place $v$ of $\QQ$ there is a pairing
$$
	U_m(\QQ_v) \times \Br U_m \to \QQ/\ZZ
$$
coming from the local invariant map $\mathrm{inv}_v \Br \QQ_v  \to \QQ/\ZZ$ from local class field theory (this is an isomorphism if $v$ is a prime number). This pairing is locally constant on the left \cite[Prop.~8.2.9]{Poo17}.

Any element $\beta \in \Br S_m$ pairs trivially on $U_m(\QQ_v)$ for almost all $v$, thus taking the sum of the local pairings gives a pairing
$$\prod_v U_m(\QQ_v) \times \Br S_m \to \QQ/\ZZ.$$
This  factors through the group $\Br S_m /\Br \QQ$ and pairs trivially with the elements of $U(\QQ)$. For  $\mathcal{B} \subseteq \Br S_m$, we let $(\prod_v U_m(\QQ_v))^{\mathcal{B}}$ be the left kernel of this pairing with respect to $\mathcal{B}$. By Corollary \ref{cor:generic}, the group $\Br S_m /\Br \QQ$ is generated by the algebra $\alpha$. Thus in our case it suffices to consider the sequence of inclusions $U_m(\QQ) \subset (\prod_v U_m(\QQ_v))^{\alpha} \subseteq \prod_v U_m(\QQ_v)$. In particular, if the latter inclusion is strict, then $\alpha$ gives an obstruction to weak approximation on $U_m$.
 
For integral points, any element $\beta \in \Br U_m$ pairs trivially on $\sU_m(\ZZ_p)$ for almost all $p$, so we obtain a  pairing $U_m(\Adele_\QQ) \times \Br U_m \to \QQ/\ZZ.$ As the local pairings are locally constant, we obtain a well-defined pairing
$$\sU_m(\Adele_\ZZ)_{\bullet} \times \Br U_m \to \QQ/\ZZ.$$
For $\mathcal{B} \subseteq \Br U_m$ we let $\sU_m(\Adele_\ZZ)_{\bullet}^{\mathcal{B}}$ be the left kernel with respect to $\mathcal{B}$, and let  $\sU_m(\Adele_\ZZ)_{\bullet}^{\Br}= \sU_m(\Adele_\ZZ)_{\bullet}^{\Br U_m}$. By Corollary \ref{cor:generic}, the map $\mathcal{A} \to \Br U_m/\Br \QQ$ is an isomorphism, hence  $\sU_m(\Adele_\ZZ)_{\bullet}^{\Br}  = \sU_m(\Adele_\ZZ)_{\bullet}^{\mathcal{A}}$. We have the inclusions $\sU_m(\ZZ) \subseteq \sU_m(\Adele_\ZZ)_{\bullet}^{\mathcal{A}}\subseteq \sU_m(\Adele_\ZZ)_{\bullet}$, so that $\mathcal{A}$  can obstruct the integral Hasse principle or strong approximation.

Let $W$ be dense Zariski open in $U_m$. As $U_m$ is smooth the set $W(\QQ_v)$ is dense in $U_m(\QQ_v)$ for all places $v$. Moreover, $\sU_m(\Zp)$ is open in $U_m(\Qp)$, hence $W(\Qp) \cap \sU_m(\Zp)$ is dense in $\sU_m(\Zp)$. As the local pairings are locally constant, to calculate the local invariants of a given element we may restrict our attention to $W$. We often take the open subset $W$ given by $(u_1^2 - 4)(u_2^2 - 4)(u_3^2 - 4) \neq 0$.

\subsection{Calculating the local invariants} \label{sec:invariants}
We now calculate the possible values for the local invariants of the elements $\alpha_{i,-}$. We do this using the following formula for the local invariants
\begin{equation} \label{eqn:Hilbert_symbol}
	\inv_p \alpha_{i,-}(\bu) = \frac{1 -(u_i - 2, m-4)_p}{4},
\end{equation}
in terms of Hilbert symbols, which holds for all $\bu \in \sU_m(\ZZ_p)$ with $u_i \neq 2$.
We begin with an elementary lemma.

\begin{lemma} \label{lem:singularities}
	Let $p \mid (m-4)$ be odd. Then the singular locus of $\mathcal{U}_m \bmod p$
	is given by the points $(2,2,2), (-2,-2,2), (2,-2,-2), (-2,2,-2).$
	
	Moreover assume that $\valp(m-4) = 1$. Then for all
	$\bu \in U_m(\ZZ_p)$, the reduction $\bu \bmod p \in 
	U_m(\FF_p)$  is a smooth $\FF_p$-point.
\end{lemma}
\begin{proof}
	The scheme $\mathcal{U}_m \bmod p$ is isomorphic to
	(an open subset of) the Cayley cubic surface.
	This is well-known to have exactly $4$ singular points which are all
	rational, and are easily verified to be the above $4$ points.
	
	For the second part, suppose that $\bu \equiv (2,2,2) \bmod p$
	(the other cases being similar). 
	Then $\valp((2u_3 - u_1u_2)^2) \geq 2$ and $\valp((u_1^2 - 4)
	(u_2^2 - 4))  \geq 2$, 
	but from the equation \eqref{eqn:GS}
	this contradicts  $\valp(m-4) =1$. 
\end{proof}

We first consider the trivial cases.

\begin{lemma} \label{lem:invp at most primes}
  Let  $i \in \{1,2,3\}$.
  Then $\inv_p \alpha_{i,-}(\bu) = 0$ for all $\bu \in \sU_m(\Zp)$ if one of the following conditions holds:
\begin{enumerate}
	\item $p \nmid 2(m-4)$.
	\item $m-4 \in \QQ_p^{*2}$.
	\item $p = \infty$ and $m > 4$.
\end{enumerate}  
\end{lemma}

\begin{proof}
 In cases $(2)$ and $(3)$, the Hilbert symbol in \eqref{eqn:Hilbert_symbol} is trivial over $\QQ_p$ as $m-4$ is a square in $\QQ_p^*$. Thus the invariant map is obviously zero in these cases.

 So we can assume that $p$ is an odd prime not dividing $m - 4$ and such that $m-4 \notin \QQ_p^{*2}$. Let $\bu \in \sU_m(\Zp)$ be such that $(u_1-2)(u_2-2)(u_3-2) \neq 0$. We claim that the $p$-adic valuation of each $u_j - 2$ is zero. Indeed, since $p \nmid (m - 4)$ and $m - 4 \bmod p$ is not a quadratic residue it follows that the left hand side of \eqref{eqn:GS} is not divisible by $p$. Hence the right hand side has zero $p$-adic valuation. But $u_j - 2$ divides the right hand side of \eqref{eqn:GS} and hence has zero $p$-adic valuation as well. As $p$ is odd, the Hilbert symbol is trivial (both entries being units).
\end{proof}

We next consider large primes which divide $m-4$ to odd valuation.

\begin{proposition} \label{prop:prolific}
	Let $p > 5$ and let $m$ be such that $\valp(m-4)$ is odd.
	Let $\mathcal{A}= \langle \alpha_{i,-} \rangle \subset \Br U_m$.
  Then the map
	$$\sU_m(\Zp) \to \Hom(\mathcal{A},\QQ/\ZZ), 
	\quad \bu \mapsto (\beta \mapsto \inv_p \beta(\bu)),$$
	induced by the Brauer--Manin pairing,	is surjective.
\end{proposition}
\begin{proof}
	It suffices to show
	that for all $(\varepsilon_1, \varepsilon_2, \varepsilon_3) \in (\ZZ/2\ZZ)^3$,
	there exists $\bu \in \sU_m(\Zp)$ such that 
	\begin{equation} \label{eqn:0_0_0}
	(\inv_p \alpha_{1,-}(\bu), \inv_p \alpha_{2,-}(\bu), \inv_p \alpha_{3,-}(\bu)) 
	= (\varepsilon_1, \varepsilon_2, \varepsilon_3).
	\end{equation}
	To prove \eqref{eqn:0_0_0},
	we claim that for all $y_1,y_2,y_3 \in \FF_p^*$
	it suffices to show the existence of an $\FF_p$-point 
	on the variety 
	$$(2u_1 - u_2u_3)^2= (u_2^2 - 4)(u_3^2 - 4), u_1 - 2 = y_1z_1^2, u_2 - 2 = y_2z_2^2,
	u_3 - 2 = y_3z_3^2$$
	which satisfies $z_1z_2z_3 \neq 0$. 
	Indeed; firstly as $(u_1 - 2)(u_2 - 2)(u_3 - 2) \neq 0$, we find 
	from Lemma \ref{lem:singularities} that
	this gives rise to a smooth $\FF_p$-point of $\sU_m$, hence a $\ZZ_p$-point
	with $\valp(u_i - 2) = 0$ by Hensel's lemma.
	Moreover $u_i - 2$ is a square in $\QQ_p$ if and only if $y_i$ is a quadratic
	residue. A simple Hilbert symbol calculation using \eqref{eqn:Hilbert_symbol}
	and the fact that $\valp(m-4)$ is odd 
	shows that we can obtain all possible choices for the invariants in \eqref{eqn:0_0_0}
	on taking all possible combinations for the $y_i$.

	To construct the given $\FF_p$-point, without loss of generality
	we may assume that $y_2 = y_3$ 
	(since $\#\FF_p^*/\FF_p^{*2} = 2$). 
	We restrict our attention to the subvariety
	given by $u_2 = u_3$. Our equations then become
	$$(2u_1 - u_2^2)^2= (u_2^2 - 4)^2, u_1 - 2 = y_1z_1^2,u_2 - 2 = y_2z_2^2.$$
	Factoring the left hand side, it suffices to solve the equations
	$$u_1 = u_2^2 - 2, u_1 - 2 = y_1z_1^2,u_2 - 2 = y_2z_2^2.$$
	This then gives the equation
	$$y_1z_1^2 = y_2^2z_2^4 + 4y_2z_2^2.$$	
	This defines a  curve  with the unique
	singular point $(z_1,z_2) = 0$. 
	The compactification of
	the normalisation is isomorphic to $\PP^1$, with at most $2$ rational
	points at infinity and at most $2$ rational points over the singular point,
	thus the affine curve has $p-2$ many $\FF_p$-points.
	Of these points at most $3$ satisfy 
	$z_1z_2 = 0$, hence providing $p - 5 > 0$, there exists
	a rational point with the required properties.
\end{proof}

For even valuation, we have the following.

\begin{proposition} \label{prop:prolific_even}
  Let $p > 5$ and let $m$ be such that $p \mid (m-4)$,
  with $\valp(m-4)$ even but $m-4 \notin \QQ_p^{*2}$.
  Then for all $\bu \in \sU_m(\Zp)$ we have
  $$\{\inv_p \alpha_{1,-}(\bu), \inv_p \alpha_{2,-}(\bu), 
  \inv_p \alpha_{3,-}(\bu)\}  \in \{ \{0,0,0\}, \{0,0,1/2\},
  \{1/2,1/2,1/2\} \}$$
  as multisets, and every such element arises for some $\bu \in \sU_m(\Zp)$.
\end{proposition}
\begin{proof}
  We first show that the local invariants must take one of the stated
  values. As $m-4 \notin \QQ_p^{*2}$ and $\valp(m-4)$ is even,
  the left hand side of \eqref{eqn:GS} has even $p$-adic valuation for any choice of $i, j, k \in \{1, 2, 3\}$. It follows that $\valp((u_j^2 - 4)(u_k^2 - 4))$ is also even for all $j,k$.
  
  If $\valp(u_j^2 - 4)$ and $\valp(u_k^2 - 4)$ are both even
  then a permutation of $i, j, k$ implies that $\valp(u_i^2 - 4)$ is even for all $i$. But one of $u_i-2$ or $u_i+2$ must be a $p$-adic unit as $p$ is odd, hence $\valp(u_i -2)$ is even for all $i$. It follows that
  all local invariants are $0$ in this case.
  
  Assume next that $\valp(u_j^2 - 4)$ and $\valp(u_k^2 - 4)$ are both odd.
  Then $\inv_p \alpha(\bu) = 1/2$. Moreover, exactly 
  one of $\valp(u_i \pm 2)$ is odd for each $i$. If $\valp(u_1 - 2)$ is odd,
  then $\inv_p \alpha_{1,-}(\bu) = 1/2$, hence 
  $\{\inv_p \alpha_{2,-}(\bu), \inv_p \alpha_{3,-}(\bu) \}  \in \{\{0,0\}, \{1/2,1/2\}\}$
  by \eqref{eqn:sum_alpha_i}. If however $\valp(u_1 + 2)$ is odd,
  then $\inv_p \alpha_{1,+}(\bu) = 1/2$ and  without loss of generality we have
  $\inv_p \alpha_{2,+}(\bu) = 1/2$ and $\inv_p \alpha_{3,+}(\bu) = 0$ 
  by \eqref{eqn:sum_alpha_i_+}. It follows that $\valp(u_1 - 2)$
  and $\valp(u_2 - 2)$ are both even and that $\valp(u_3 - 2)$ is odd,
  hence we obtain the local invariants $0,0,1/2$, as claimed. 
  
  We now show that these possibilities for the local invariants
  can actually be realised. 
  A similar argument to the proof of Proposition \ref{prop:prolific}
  shows the existence of $\bu$
  such that all the $u_i-2$ are units, hence the local invariants
  are all $0$ in this case (this is the only point in the proof we use that
  $p \neq 3,5$).
  
  For the non-trivial local invariants, we will construct $p$-adic
  solutions with given valuations.
  Write $m-4 = p^wz$, where $z \in \ZZ_p^*$,
  and make the change of variables 
  $pv_1 = u_1 - 2$ and $p^{w-1}v_2 = u_2 - 2$.
  Then the equation \eqref{eqn:GS} becomes
  $$(2u_3 - (pv_1 + 2)(p^{w-1}v_2 + 2))^2  - 4zp^w 
  = p^wv_1v_2(pv_1 + 4)(p^{w-1}v_2 + 4).$$
  We make the change of variables $2u_3 - (pv_1 + 2)(p^{w-1}v_2 + 2) = p^wv_3$
  and cancel  $p^w$ to obtain
  $$p^{w}v_3^2  - 4z = v_1v_2(pv_1 + 4)(p^{w-1}v_2 + 4).$$
  Modulo $p$, this becomes
  $$z \equiv - 4v_1v_2 \bmod p.$$
  This clearly has the non-singular solution in units
  $v_1 \equiv 1 \bmod p$ and $v_2 \equiv -4^{-1}z \bmod p.$
  Thus we may apply Hensel's lemma to show the existence
  of a $p$-adic solution with $v_p(u_1-2)$ and $v_p(u_2 - 2)$
  both odd. For such a solution we have 
  $\inv_p \alpha_{1,-}(\bu) =  \inv_p \alpha_{2,-}(\bu) = 1/2$ as 
  $m-4 \notin \QQ_p^{*2}$ and $w$ is even.
  This realises the possibilities $\{1/2,1/2,1/2\}$ by the first part of 
  the lemma.
  
  A similar argument shows the existence of a $p$-adic solution with 
  $v_p(u_1-2)$ and $v_p(u_2 + 2)$ both odd.
  This is easily 
  seen to realise the local invariants  $\{1/2,0,0\}$, as required.
\end{proof}

We now show that the conclusion of Proposition \ref{prop:prolific}
does not hold for $p=3,5$, so that Proposition \ref{prop:prolific} is sharp.

\begin{proposition} \label{prop:3_5}
	Let $p=3,5$ and let $m$ such that $\valp(m-4)=1$.
	Let $\mathcal{A}= \langle \alpha_{i,-} \rangle \subset \Br U_m$. 
	\begin{enumerate}
		\item If $p=3$, then for all $\bu \in \sU_m(\Zp)$ 	we have
		$$\{\inv_p \alpha_{1,-}(\bu), \inv_p \alpha_{2,-}(\bu), 
		\inv_p \alpha_{3,-}(\bu)\} = \{0,0,1/2\}$$
		as multisets.
		\item If $p= 5$, then the image of map induced by the Brauer pairing
		$$\sU_m(\Zp) \to \Hom(\mathcal{A},\QQ/\ZZ)$$
		contains every element
		except the trivial homomorphism.
	\end{enumerate}
\end{proposition}
\begin{proof}
	By Lemma \ref{lem:singularities} the reduction modulo $p$
	of each $\Zp$-point is smooth.
	
	$p = 3$:  One checks that up to the $S_3$-action, the only smooth rational
	points modulo $3$ are $(0, 0, 1)$ and $(0, 0, 2)$. Let $\bu$ be a $3$-adic
	lift of one of these points. We first consider
	the element $\alpha = (u_1^2 - 4,m-4)$. As $-4$ is not a quadratic
	residue modulo $3$, we find that $\inv_3 \alpha(\bu)= 1/2$.
	Moreover as $u_1-2$ and $u_2 -2$ are quadratic residues
	for both points, we find that $\inv_3 \alpha_{1,-}(\bu) = \inv_3 \alpha_{2,-}(\bu) = 0$.
	Combining this with the relation \eqref{eqn:sum_alpha_i}
	shows that $\inv_3 \alpha_{3,-}(\bu) = 1/2$. Considering
	permutations proves the result.
	
	$p = 5$: Again, up to the $S_3$-action, the only smooth rational
	points modulo $5$ are $(0, 0, 2)$, $(0, 0, 3)$, $(1, 1, 2)$, 
	$(1, 1, 4)$, $(1, 3, 4)$, $(2, 2, 4)$ and $(4, 4, 4)$. 
	One realises the invariants 
	$$\inv_5 \alpha_{1,-}(\bu) = 0, \inv_5 \alpha_{2,-}(\bu) = 0, \inv_5 \alpha_{3,-}(\bu) = 1/2,$$
	by considering $5$-adic lifts of the point  $(1, 1, 4)$. The other
	non-trivial combinations are obtained by considering permutations of
	the points 	$(0,0,3)$ and $(4,4,4)$. To show that the trivial homomorphism
	is not realised, we can ignore those points where at least one of
	$u_i -2$ is not a square. This leaves the point $(1,1,2)$. Here 
	we have $\inv_5 \alpha(\bu) = 1/2$ as $u_1^2 - 4 \equiv 2 \bmod 5$ is a quadratic
	non-residue, whence $\inv_5 \alpha_{3,-}(\bu) = 1/2$ by \eqref{eqn:sum_alpha_i}. So the trivial homomorphism is not realised.
\end{proof}

We continue with the study of the local invariant map when $p = 2$.

\begin{lemma} \label{lem:2}
  Let $m \in \ZZ$ be such that $m \nequiv 3 \bmod 4$. If $m \nequiv 1 \bmod 8$, then there exists $\bu \in \sU_m(\ZZ_2)$ such that 
  \begin{equation*}
    (\inv_2 \alpha_{1, -}(\bu), \inv_2 \alpha_{2, -}(\bu) \inv_2 \alpha_{3, -}(\bu)) = 
    \begin{cases}
      (0, 0, 0) &\mbox{if } \valn_2(m - 4) \equiv 0 \bmod 2, \\  
      (0, 0, 1/2) &\mbox{if } \valn_2(m - 4) \equiv 1 \bmod 2.
    \end{cases}
  \end{equation*}
  Alternatively, if $m \equiv 1 \bmod 8$, then for every $\bu \in \sU_m(\ZZ_2)$ we have
  \begin{equation*}
    \{\inv_2 \alpha_{1, -}(\bu), \inv_2 \alpha_{2, -}(\bu), \inv_2 \alpha_{3, -}(\bu)\} = \{0, 1/2, 1/2\}
  \end{equation*}
  as multisets.
\end{lemma}

\begin{proof}
	As $m \nequiv 3 \bmod 4$ we have $\sU_m(\ZZ_2) \neq \emptyset$ by
	\cite[\S6.3]{GS17}.
  
  Firstly, assume that $2 \mid m$. Then $2 \nmid (3 - m)$ and clearly $(3, 3, 3 - m)$ satisfies the reduction of \eqref{def:Um} $\bmod 8$. By Hensel's lemma this lifts to a point $\bu \in \sU_m(\ZZ_2)$. Moreover, $u_i^2 - 4 \equiv 5 \bmod 8$ and thus 
  \begin{equation*}
    (u_i^2 - 4, m - 4)_2
    = (-1)^{\valn_2(m - 4)}.
  \end{equation*}
  By \eqref{eq:alpha different reps} and \eqref{eqn:Hilbert_symbol} we have $\inv_2 \alpha(\bu) = 1/2$ if $\valn_2(m)$ is odd or $0$ otherwise.

  On the other hand, $(3 - 2, m - 4)_2 = 1$ and thus the local invariant maps for $\alpha_{1, -}$ and $\alpha_{2, -}$ are trivial. What is left is to take into account the relation \eqref{eqn:sum_alpha_i} which implies that the local invariant map for $\alpha_{3, -}$ is equal to $0$ if $\valn_2(m - 4)$ is even and $1/2$ otherwise. 

  Assume now that $2 \nmid m$, that is $m \equiv 1 \bmod 4$ by our assumption on $m$. If $m \equiv 5 \bmod 8$, then $m - 4 \in \QQ_2^{*2}$ and hence $\inv_2\alpha_{i, -}(\bu) = 0$  by Lemma \ref{lem:invp at most primes}.

  Finally, if $m \equiv 1 \bmod 8$, then the right hand side of \eqref{def:Um} is odd and thus for each $\bu \in \sU_m(\ZZ_2)$ exactly one its coordinates is a $2$-adic unit, $u_i$ say. Moreover the other two coordinates $u_j$ and $u_k$ must be divisible by $4$. Hence $\valn_2(u_i - 2) = 0$ while $\valn_2(u_j - 2) = \valn_2(u_k - 2) = 1$. Since $m - 4 \equiv 5 \bmod 8$, a simple Hilbert symbol calculation and \eqref{eqn:Hilbert_symbol} imply our claim. This concludes the proof.
\end{proof}

Lastly, we show that the local invariant for $\alpha$ at $\infty$ is trivial for most choices of $m$.

\begin{lemma} \label{lem:infinity}
  Let $m \in \RR \setminus [-8, 4]$. Then
  $\inv_\infty \alpha(\bu) = 0$
  for all $\bu \in U_m(\RR)$.
\end{lemma}
\begin{proof}
	If $m - 4 > 0$ then $m- 4 \in \RR^{*2}$ and the result follows from Lemma~\ref{lem:invp at most primes}. If $m - 4 < 0$ then the left hand side of \eqref{eqn:GS} is positive. Hence $u_i \in (-2, 2)$, $i = 1, 2, 3$ or each $u_i$ lies outside this interval. In the latter case we clearly have $u_1^2 - 4 > 0$ so the local invariant is trivial. It therefore suffices to show that the former case does not occur under our assumptions. But we have $u_1^2 + u_2^2 + u_3^2 - u_1u_2u_3 \ge - u_1u_2u_3 \ge -8$ in the box $[-2, 2]^3$. As $m < -8$, this therefore cannot equal $m$, thus there are no such real points, as required.
\end{proof}

\subsubsection{Consequences for the Brauer-Manin obstruction}
We now use the above results on the possibilities for the local invariants to study the  Brauer-Manin obstruction. We still assume that $m$ satisfies the hypothesis of Corollary~\ref{cor:generic}.

We first consider weak and strong approximation. (See the introduction for the notation $\overline{\sU_m(\ZZ)}$ and $\overline{\sU_m(\QQ)}$.)

\begin{corollary} \label{cor:SA}
	Assume there exists a prime $p >3$ with $\valp(m-4)$ odd.
	Then 
	\begin{equation}
	\overline{U_m(\QQ)} \neq \prod_{v} U_m(\QQ_v).
	\label{eqn:WA}
	\end{equation}
	Moreover, if $\sU_m(\Adele_\ZZ) \neq \emptyset$ then
	\begin{equation}
	\overline{\sU_m(\ZZ)} \neq \sU_m(\Adele_\ZZ)_{\bullet} \cap \overline{U_m(\QQ)},
	\label{eqn:SA}
	\end{equation}
	i.e.~$\sU_m$ fails weak approximation and has a failure of strong approximation which is not explained
	by the failure of weak approximation.
\end{corollary}

\begin{proof}
	We only prove \eqref{eqn:SA}, as a similar argument shows that $\alpha$ gives an obstruction to weak approximation in this case.
	
	To prove \eqref{eqn:SA}, we note from \cite{SS91} that the Brauer-Manin obstruction is the only one to 
	weak approximation on $S_m$. But the Brauer group of $S_m$ modulo constants
	is generated by $\alpha$ by Lemma \ref{lem:Br of comp}.
	In particular, it follows that
	$$
		\overline{U_m(\QQ)} = (\prod_{v} U_m(\QQ_v))^{\alpha}.
	$$
	Hence to prove \eqref{eqn:SA}, it suffices to show that
	$\sU_m(\Adele_{\ZZ})^{\mathcal{A}} \neq \sU_m(\Adele_{\ZZ})^{\alpha}$.
	To show this, let $\bu \in \sU_m(\Adele_{\ZZ})$. 
	By Propositions \ref{prop:prolific}~and~\ref{prop:3_5} and \eqref{eqn:sum_alpha_i}, on $\sU_m(\ZZ_p)$
	the tuple $(\inv_p \alpha, \inv_p \alpha_{1,-})$ takes all possible
	values of $(\ZZ/2\ZZ)^2$.
	Hence there exists $\bu_p \in \sU_m(\ZZ_p)$ such that 
	$\inv_p \alpha(\bu_p) = -\sum_{v \neq p} \inv_v \alpha(\bu)$
	but $\inv_p \alpha_{1,-}(\bu_p) \neq -\sum_{v \neq p} \inv_v \alpha_{1,-}(\bu)$.
	Then the adele $\bu'$ given by replacing the $p$th part of $\bu$ by $\bu_p$
	satisfies $\bu' \in \sU_m(\Adele_{\ZZ})^{\alpha}$ but 
	$\bu' \notin \sU_m(\Adele_{\ZZ})^{\alpha_{1,-}}$, as required.
\end{proof}

We next consider the integral Hasse principle.

\begin{corollary} \label{cor:HP}
	Suppose that there exists a prime $p>5$ with $p \mid (m-4)$ such that 
	$m - 4 \notin \QQ_{p}^{*2}$.
	Then there is no Brauer--Manin obstruction to the integral Hasse principle for $\sU_m$.	
\end{corollary}
\begin{proof}
  We can assume that $\sU_m(\Adele_\ZZ) \neq \emptyset$, otherwise there is nothing to prove.
	We begin with the case where $\valp(m-4)$ is odd. There is a standard
	argument for deducing the result from Proposition \ref{prop:prolific},
	which we recall for completeness.
	It suffices to show that $\sU_m(\Adele_\ZZ)^{\mathcal{A}} \neq \emptyset$.
	Let $\bu \in \sU_m(\Adele_\ZZ)$. By Proposition \ref{prop:prolific},
	we find that there exists $\bu_p \in \sU_m(\ZZ_p)$ such that 
	$\inv_p \beta(\bu_p) = -\sum_{v \neq p} \inv_v \beta(\bu)$ for all $\beta \in \mathcal{A}$.
	Then the adele $\bu'$ given by replacing the $p$th part of $\bu$ by $\bu_p$
	satisfies $\bu' \in \sU_m(\Adele_\ZZ)^{\mathcal{A}}$, as required.

	We now consider the case where $\valp(m-4)$ is even and $m - 4 \notin \QQ_{p}^{*2}$.
	Here the argument is more involved as we need to consider more than one prime. 
	We may assume that
	we are not in the previous case, so that
	$m = 4 \pm n d^2$ for some positive $n, d \in \ZZ$ such that the prime divisors of $n$ are $\le 5$ while those of $d$ are $> 5$, and that $m - 4 \notin \QQ_p^{*2}$ for some $p \mid d$.
  
  Let $\bu \in \sU_m(\Adele_\ZZ)$ and let $\varepsilon_i = \sum_{v \neq p} \inv_v \alpha_{i,-}(\bu)$.
	We use our information about the possibilities for the local 
	invariants at $p$ given in Proposition \ref{prop:prolific_even}. A similar argument to the one above verifies that we can modify $\bu$ at $p$ so that $\bu \in \sU_m(\Adele_\ZZ)^{\mathcal{A}}$ unless $\{\ve_1, \ve_2, \ve_3\} = \{0, 1/2, 1/2\}$ as multisets. But, if there was a prime $q \mid d$ with $q \neq p$ and $m - 4 \notin \QQ_q^{\ast 2}$, then by Proposition~\ref{prop:prolific_even} we can  modify $\bu$ at $q$ so that $\{\ve_1, \ve_2, \ve_3\} \neq \{0, 1/2, 1/2\}$. We can thus assume that $p$ is the only prime dividing $d$ for which $m - 4 \notin \QQ_{p}^{*2}$. 
  To complete the proof, we will show that $\bu$ can be modified at $2$, $3$ and $5$ or at $\infty$ so that $\{\ve_1, \ve_2, \ve_3\} \neq \{0, 1/2, 1/2\}$.

  Firstly, we analyse the case $m - 4 < 0$. Without loss of generality assume that $(\ve_1, \ve_2, \ve_3) = (0, 1/2, 1/2)$. Since $p \mid d$ and $p > 5$ we have $m \le 4 - 7^2 < -8$ and thus as in the proof of Lemma~\ref{lem:infinity} we see that $|u_i| > 2$, $i = 1, 2, 3$ for any real point $(u_1, u_2, u_3)$ on $\sU_m$. The relation \eqref{eqn:sum_alpha_i} and Lemma~\ref{lem:infinity} imply that the local invariant maps for $\alpha_{i,-}(\bu)$ at infinity belong to $\{\{0, 0, 0\}, \{0, 1/2, 1/2\}\}$. Replacing the real component $(u_1, u_2, u_3)$ of $\bu$ by $(u_1, -u_2, -u_3)$, which is again a real point on $\sU_m$, shows that both possibilities for the local invariant maps at $\infty$ are realised. By doing so we modify $\bu$ at $\infty$ so that $(\ve_1, \ve_2, \ve_3) = (0, 0, 0)$, as required.
  
  We can now assume that $m - 4 = nd^2 > 0$. Since $m - 4 \notin \QQ_{p}^{*2}$ then $m \neq 4 + d^2$ and hence $n > 1$. Thus at least one of $2$, $3$ or $5$ must divide $n$. We will modify $\bu$ at $2$, $3$ and $5$ in a way such that $\ve_1 = \ve_2 = 0$. This clearly implies $\{\ve_1, \ve_2, \ve_3\} \neq \{0, 1/2, 1/2\}$. By Lemma~\ref{lem:invp at most primes} we have $\sum_{v \neq 2, 3, 5, p} \inv_v \alpha_{i,-}(\bu) = 0$ for $i = 1, 2$. Therefore it suffices to show that
  \begin{equation} \label{eq:inv for 1 and 2}
      \sum_{v \in \{2, 3, 5\}} \inv_v \alpha_{i, -}(\bu) = 0, \quad i = 1, 2.
  \end{equation}

  We keep the $3$-adic component of $\bu$ as is if $3 \nmid n$. In this case $\inv_3\alpha_{i, -}(\bu) = 0$ for $i = 1, 2$ by Lemma~\ref{lem:invp at most primes}. Alternatively, if $3 \mid n$, then we replace the $3$-adic component of $\bu$ by a $3$-adic lift of the smooth $\FF_3$-point $(0, 0, 1)$. A simple calculation of Hilbert symbols and \eqref{eqn:Hilbert_symbol} show that $\inv_3 \alpha_{i, -}(\bu) = 0$ for $i = 1, 2$ regardless of the parity of the $3$-adic valuation of $n$. Thus no matter if $3 \mid n$ or $3 \nmid n$ we get
  \begin{equation} \label{eq:inv at 3}
    \inv_3 \alpha_{i, -}(\bu) = 0, \quad i = 1, 2.
  \end{equation}
  For the primes $2$ and $5$ we follow different approaches depending on the residue of $m \bmod 8$. 

  If $m \nequiv 1 \bmod 8$ we do the following. We replace the $2$-adic component of $\bu$ with a point in $\sU_m(\ZZ_2)$ for which $\inv_2\alpha_{i, -}(\bu) = 0$ for $i = 1, 2$.
  By Lemma~\ref{lem:2} there is such a $2$-adic point. If $5 \mid n$ we replace the $5$-adic component of $\bu$ by a $5$-adic lift of the smooth $\FF_5$ point $(1, 1, 4)$. Once more a simple Hilbert symbol calculation shows that $\inv_5 \alpha_{i, -}(\bu) = 0$ for $i = 1, 2$ regardless of the parity of the $5$-adic valuation of $n$. Lastly, if $5 \nmid n$ we keep the $5$-adic component of $\bu$ as is, in which case the local invariant maps for $\alpha_{1, -}$ and $\alpha_{2, -}$ at $5$ are zero by Lemma~\ref{lem:invp at most primes}. We conclude that our modifications in this case yield
  \begin{equation*}
      \inv_2 \alpha_{i, -}(\bu) = \inv_5 \alpha_{i, -}(\bu) = 0, \quad i = 1,2.
  \end{equation*}
  This and \eqref{eq:inv at 3} clearly imply \eqref{eq:inv for 1 and 2}.

  Assume now that $m \equiv 1 \bmod 8$. Here we replace the $2$-adic component of $\bu$ with a point in $\sU_m(\ZZ_2)$ for which the local invariant maps for $\alpha_{i, -}$ at 2 are equal to $(1/2, 1/2, 0)$. We can do so by Lemma~\ref{lem:2}. We have $m = 4 + nd^2 \equiv 4 + n \bmod 8$, hence $n$ must be congruent to $5 \bmod 8$. Since the only primes which divide $n$ in this case are $3$ and $5$ this is only possible if $\valn_5(n)$ is odd while $\valn_3(n)$ is even. We replace the $5$-adic component of $\bu$ by a $5$-adic lift of the smooth $\FF_5$-point $(0, 0, 3)$. Therefore
  \begin{equation*}
      \inv_2 \alpha_{i, -}(\bu) = \inv_5 \alpha_{i, -}(\bu) = \frac{1}{2}, \quad i = 1, 2.
  \end{equation*}
  Once more the above and \eqref{eq:inv at 3} imply \eqref{eq:inv for 1 and 2}. This concludes the proof.
\end{proof}

\subsection{Examples of Brauer--Manin obstruction} \label{sec:BM_examples}
We now use the calculations from the previous section to give examples of a Brauer--Manin obstruction to the integral Hasse principle and strong approximation.
For each prime $p$ including $\infty$ let $\alpha_p(\bu)$ denote the Hilbert symbol corresponding to $\alpha(\bu)$. Our results are inspired by the results in \cite[\S 8]{GS17}; in particular we show how the counter-examples to the integral Hasse principle given in \cite[Prop.~8.1(i), 8.2, 8.3]{GS17} can be explained via the Brauer--Manin obstruction. 

\begin{proposition} \label{prop:alpha GS8.1i}
  Let $m = 4 \pm 2d^2$ where $d$ is a positive integer whose prime divisors are congruent to $\pm 1 \bmod 8$ if $m - 4 > 0$ and to $1$ or $3 \bmod 8$ if $m - 4 < 0$. If $m = 4 + 2d^2$ assume that $d \bmod 9 \in \{0, \pm 3, \pm 4\}$. Then there is a Brauer--Manin obstruction to the integral Hasse principle for $\sU_m$.
\end{proposition} 

\begin{proof}

  In the case $m = 4 + 2d^2$ the assumption on the residue of $d \bmod 9$ ensures that $\sU_m(\Adele_\ZZ) \neq \emptyset$ while if $m = 4 - 2d^2$ then $\sU_m(\Adele_\ZZ) \neq \emptyset$ for any $d$, as is verified using \cite[Prop.~6.1]{GS17}. 
  We will show that for each point $\bu \in \sU_m(\Zp)$ we have
  \begin{equation*}
    \inv_p \alpha(\bu) =
    \begin{cases}
      1/2 &\mbox{if } p = 2, \\
      0 &\mbox{otherwise.}    
    \end{cases}
  \end{equation*}
  This shows that $\alpha$ gives rise to a Brauer--Manin obstruction to the 
  integral Hasse principle.
 
 For each prime $p \le \infty$ and  each $\bu \in \sU_m(\Zp)$ with $u_i^2 \neq 4$, by \eqref{eq:alpha different reps} we have
  \begin{equation*}
    \alpha_p(\bu) 
    = (u_i^2 - 4, \pm 2d^2)_p = (u_i^2 - 4, \pm 2)_p ,
     \quad i \in \{1, 2, 3\}.
  \end{equation*}
  We make use of the relation \eqref{eqn:sum_alpha_i}. If $p \nmid 2d$ the correctness of our claim follows from case $(1)$ of Lemma~\ref{lem:invp at most primes}. On the other hand, if $p \mid d$, then the congruence restrictions modulo $8$ on $p$ imply that the second argument of the Hilbert symbol is a square in $\Qp^\ast$. By case $(2)$ of Lemma~\ref{lem:invp at most primes} we are done. If $m - 4 > 0$ our statement holds for $p = \infty$ by Lemma~\ref{lem:infinity}. Alternatively, if $m - 4 < 0$ then the assumption on the residue of $d \bmod 9$ implies that $d > 1$. Moreover, the least prime that can divide $d$ is $3$ and hence $m \le -14$. The result for $p = \infty$ now follows from Lemma \ref{lem:infinity}.

  We can now assume that $p = 2$. For each point $\bu \in \sU_m(\ZZ_2)$ we have $2 \nmid \bu$. Indeed, if $2 \mid \bu$, then the left hand side of \eqref{def:Um} would have been congruent to $0$ or $4 \bmod 8$, while the right hand is congruent to $2$ or $6 \bmod 8$ since $d$ is odd and $d^2 \equiv 1 \bmod 8$; contradiction. Let $u_i$ be the coordinate of $\bu$ which is a $2$-adic unit. Then $u_i^2 - 4 \equiv 5 \bmod 8$ and since $\pm 2d^2 \equiv \pm 2 \bmod 8$ we have
  \begin{equation*}
    \alpha_2(\bu) 
    = (5, \pm 2)_2
    = -1.
  \end{equation*} 
  By \eqref{eqn:Hilbert_symbol} this concludes the proof of Proposition~\ref{prop:alpha GS8.1i}. 
\end{proof}

\begin{remark} \label{rem:GS_HP_holds}
  The Hasse principle need not fail if there is a prime divisor of $d$ which does not satisfy the assumptions imposed in the statement of Proposition~\ref{prop:alpha GS8.1i}. For $m = 4 - 2\cdot 5^2 = -46$, one easily checks that the equation \eqref{def:Um} has the integer solution $(3, 7, 8)$. Moreover, for $m = 4 - 2\cdot 7^2 = -94$ we have $(5,5,9) \in \sU_m(\ZZ)$. Alternatively, for $m = 4 + 2d^2$ we have $(-1, 1, 4) \in \sU_{4 + 2\cdot 3^2 }(\ZZ)$ and $(-3, 3, 3) \in \sU_{ 4 + 2\cdot 5^2}(\ZZ)$.
\end{remark}

\begin{proposition} \label{prop:alpha GS8.2}
  Let $m = 4 + 12d^2$ where $d$ is an odd integer whose prime divisors are congruent to $\pm 1 \bmod 12$ and such that $d^2 \equiv 25 \bmod 32$. Then there is a Brauer--Manin obstruction to the integral Hasse principle for $\sU_m$.
  \end{proposition} 

\begin{proof}
	We will show that for each point $\bu \in \sU_m(\Zp)$ we have
  \begin{equation*}
    \inv_p \alpha(\bu) =
    \begin{cases}
      1/2 &\mbox{if } p = 3, \\
      0 &\mbox{otherwise,}    
    \end{cases}
  \end{equation*}
  so that $\alpha$ again gives an obstruction to the Hasse principle.

  Once more, the assumption $m = 4 + 12d^2$ implies that $\sU_m(\Adele_{\ZZ}) \neq \emptyset$ \cite[Prop.~6.1]{GS17}. For each  $p \le \infty$ and  each $\bu \in \sU_m(\Zp)$ with $u_i^2 \neq 4$, by \eqref{eq:alpha different reps} we have
  \begin{equation} \label{eq:alpha8.2}
    \alpha_p(\bu) 
    = (u_i^2 - 4, 12d^2)_p, \quad i \in \{1, 2, 3\}.
  \end{equation}

  If $p \nmid 6d^2$ our claim follows from case $(1)$ of Lemma~\ref{lem:invp at most primes}. If $p \mid d$, then $p \equiv \pm 1 \bmod 12$ and hence $3 \in {\Qp^\ast}^2$. Thus $12d^2 \in {\Qp^\ast}^2$ and case $(2)$ of Lemma~\ref{lem:invp at most primes} implies our claim. Finally, $m - 4 > 0$ and case $(3)$ of Lemma~\ref{lem:invp at most primes} is applicable for $p = \infty$. It remains to examine $p = 2, 3$.
  
  For $p=3$, the relation \eqref{eqn:sum_alpha_i} and Proposition \ref{prop:3_5} 
  shows the claim $\inv_3(\alpha(\bu)) = 1/2$ for all $\bu \in \sU_m(\Zp)$.

  Assume now that $p = 2$. Let $n_i = \valn_2(u_i^2 - 4)$ and let $\omega(u_i^2 - 4)$
  be $(u_i^2 - 4)/2^{n_i} \bmod 8$. Since $d$ is odd we have $d^2 \equiv 1 \bmod 8$ and thus by \eqref{eq:alpha8.2} we have
  \begin{equation*}
    \alpha_2(\bu)
    = (-1)^{\frac{\omega\(u_i^2 - 4\) - 1}{2} + n_i}.
  \end{equation*}
  Let $\bu \in \sU_m(\ZZ_2)$ be a point having one of its coordinates a $2$-adic unit, $u_i$ say. Then $u_i^2 \equiv 1 \bmod 8$, hence $n_i = 0$ and $\omega(u_i^2 - 4) = 5$. Therefore $\alpha_2(\bu) = 1$. 

  Arguing in the same way as in the proof of \cite[Prop.~8.2.]{GS17} one easily verifies that there are no other points in $\sU_m(\ZZ_2)$. Indeed, let $\bu \in \sU_m(\ZZ_2)$ have all of its coordinates divisible by two. Then we can write $\bu = 2\by$, in which case the equation defining $\sU_m$ takes the shape
  \begin{equation} \label{eq:Um with y}
    y_1^2 + y_2^2 + y_3^2 - 2y_1y_2y_3 = 1 + 3d^2,
  \end{equation}
  or alternatively 
  \begin{equation*}
    (y_3 - y_1y_2)^2 - 3d^2 = (y_2^2 - 1)(y_3^2 - 1),
  \end{equation*}
  
  If one of $y_1, y_2, y_3$ was a unit in $\ZZ_2$, then the last equation above would imply that $3$ is a quadratic residue modulo 8, a contradiction. On the other hand the reduction of \eqref{eq:Um with y} modulo $8$ implies that $4$ cannot divide all of the $y_i$'s simultaneously. Thus there are only three possibilities, $4$ divides two of the coordinates of $\bu$, one of them or none of them. The reduction of \eqref{eq:Um with y} modulo 16 rules out the first two cases since the right hand side is $12 \bmod 16$ by the assumption on $d$. At the same time the left hand side is $4$ and $8 \bmod 16$, respectively. In the remaining possibility we have $y_1 \equiv y_2 \equiv y_3 \equiv 2 \bmod 4$. Then \eqref{eq:Um with y} becomes
  \begin{equation*}
    z_1^2 + z_2^2 + z_3^2 - 4z_1z_2z_3 = \frac{1 + 3d^2}{4},
  \end{equation*}
  with $2 \nmid z_1z_2z_3$. It is clear that the left hand side is congruent to $7 \bmod 8$ while the right hand side is $3 \bmod 8$ since $d^2 \equiv 25 \bmod 32$: a contradiction. This concludes the proof of Proposition~\ref{prop:alpha GS8.2}.
\end{proof}

The following is a generalisation of \cite[Prop.~8.3]{GS17} (\emph{loc.~cit.} restricts to those $d$ whose prime divisors are congruent to $\pm 1 \bmod 20$).

\begin{proposition} \label{prop:alpha GS8.3}
  Let $m = 4 + 20d^2$ where $d \equiv \pm 4 \bmod 9$ is an odd integer whose prime divisors are congruent to $\pm 1 \bmod 5$. Then there is a Brauer--Manin obstruction to the  integral Hasse principle on $\sU_m$.
\end{proposition} 

\begin{proof}
We have $\sU_m(\Adele_\ZZ) \neq \emptyset$ \cite[Prop.~6.1]{GS17}. Our strategy this time requires a more detailed analysis on the local invariant maps for the $\alpha_{i, -}$.
We first show that for each point $\bu \in \sU_m(\Zp)$ we have
  \begin{equation} \label{eqn:p<>5}
    \inv_p \alpha_{i,-}(\bu) =
      0 \quad \mbox{if } p \neq 5.
  \end{equation}
  We explain at the end how to get a Brauer--Manin obstruction to the integral Hasse principle.
  
 We once more restrict attention $\bu \in \sU_m(\Zp)$ with $u_i^2 \neq 4$. 
As in Proposition~\ref{prop:alpha GS8.2}, we deal with primes $p \neq 2, 5$ with the help of cases $(1)$, $(2)$ and $(3)$ of Lemma~\ref{lem:invp at most primes}. On the other hand, for each $i = 1, 2, 3$ we have 
\begin{equation} \label{eq:alpha8.3}
  (u_i - 2, 20d^2)_2
  = (u_i - 2, 5)_2
  =
   (-1)^{\valn_2(u_i - 2)}.
\end{equation}
Assume  there is $\bu \in \sU_m(\ZZ_2)$ with $2 \mid \bu$. Following the proof of \cite[Prop.~8.3]{GS17} we make the change of variables $y_i = u_i/2$ for all $i = 1, 2, 3$. Thus
\begin{equation*}
  y_1^2 + y_2^2 + y_3^2 - 2y_1y_2y_3
  = 1 + 5d^2.
\end{equation*}
If $2$ divided all $y_i$ then the left hand side of the last equation above would have been congruent to $0 \bmod 4$ while the right hand side would have been congruent to $2 \bmod 4$. A contradiction! Assume that $2 \nmid y_1$ which implies that $y_1^2 \equiv 1 \bmod 8$. We can rewrite the equation defining $\sU_m$ in the new variables in the following way
\begin{equation*}
  (y_3 - y_1y_2)^2 - 5d^2 = (y_1^2 - 1)(y_2^2 - 1).
\end{equation*}
The right hand side is congruent to $0 \bmod 8$. This implies that either $5$ is a square modulo $8$ or $2 \mid d$. Both are false. 

Thus for all $\bu \in \sU_m(\ZZ_2)$ least one of the coordinates of $\bu$ a 2-adic unit. However, $m$ is even and then the reduction of \eqref{def:Um} modulo $2$ implies that at least two of the coordinates of $\bu$ must be $2$-adic units, $u_1$ and $u_2$ say. By \eqref{eq:alpha8.3} the Hilbert symbols for $\alpha_{1, -}(\bu)$ and $\alpha_{2, -}(\bu)$ are equal to $1$. Moreover, $u_1^2 - 4 \equiv 5 \bmod 8$ and thus by \eqref{eq:alpha different reps} we have
\begin{equation*}
  \alpha_2(\bu)
  =(5, 5)_2
  = 1.
\end{equation*}
We conclude that the local invariant map for $\alpha(\bu)$ is zero. Since all local invariant maps except possibly the one for $\alpha_{3, -}(\bu)$ vanish, then the relation \eqref{eqn:sum_alpha_i} implies that $\inv_p \alpha_{3, -}(\bu)$ must be zero as well.
This proves \eqref{eqn:p<>5}.

We now consider $p = 5$. Let $\bu \in \sU_m(\ZZ_5)$. Then by Proposition \ref{prop:3_5}, there exists $i$ (depending on $\bu$) such that $\inv_5 \alpha_{i,-}(\bu) = 1/2$. It follows from this and \eqref{eqn:p<>5} that for every adelic point, there is some $\alpha_{i,-}$ for which the sum of all local invariants is equal to $1/2$. Hence $\sU_m(\Adele_\ZZ)^{\Br} = \emptyset$ as required.
\end{proof}

\begin{remark}
	In Propositions \ref{prop:alpha GS8.1i} amd \ref{prop:alpha GS8.2}, we obtained
	a Brauer--Manin obstruction by considering the single element $\alpha$,
	i.e.~we showed $\sU_m(\Adele_\ZZ)^{\alpha} = \emptyset$.
	
	In Proposition \ref{prop:alpha GS8.3} however, the obstruction is more
	complicated. Namely, a simple application of Proposition \ref{prop:3_5}
	shows that $\sU_m(\Adele_\ZZ)^{\beta} \neq \emptyset$ for all $\beta \in \Br U_m$,
	so to obtain an obstruction one really needs to consider the whole Brauer group,
	as we have done.
\end{remark}

We finish with an explicit example of a surface which fails weak approximation, but also has a failure of strong approximation which is \emph{not} explained by the failure of weak approximation (many similar examples can be easily constructed using Corollary \ref{cor:SA}).
\begin{proposition} \label{prop:fail_SA}
	Let $m = 4 + 41$. Then $\sU_m(\ZZ)$ is Zariski dense. But:
	\begin{enumerate}
		\item For all $\bu \in \sU_m(\QQ)$
		we have $\bu \bmod 41 \not \equiv (1,1,2) \in \sU_m(\FF_{41})$.
		\item There exists $\bu \in \sU_m(\QQ)$
	with $\bu \bmod 41 \equiv (0,5,15) \in \sU_m(\FF_{41})$,
	but for every integer point $\bu \in \sU_m(\ZZ)$
	we have $\bu \bmod 41 \not \equiv (0,5,15)$.
	\end{enumerate}
\end{proposition}
\begin{proof}
	One verifies that $\sU_m(\ZZ) \neq \emptyset$, e.g.~$(0,3,6) 
	\in \sU_m(\ZZ)$. Zariski density now follows from \cite[(1.5)]{GS17}.
	
	(1) A minor variant of  Lemma \ref{lem:invp at most primes} shows that
	for all $\bu \in \sU(\QQ_p)$ we have
  \begin{equation} \label{eqn:p<>41_1}
    \inv_p \alpha(\bu) =
      0 \quad \mbox{if } p \neq 41
  \end{equation}
  (for $p=2$ we note that $41 \equiv 1 \bmod 8$ is a square in $\QQ_2^*$).
  At $41$ one can use Corollary \ref{cor:SA} to see that
  $U_m(\QQ)$ is not dense in $U_m(\QQ_{41})$. To make this explicit,
  following the proof  of Proposition \ref{prop:prolific},
  it suffices find a point $\bu \in \sU(\FF_{41})$ such that
  $u_1^2 -4$ is a non-zero quadratic non-residue. One easily
  verifies that $(1,1,2) \in \sU(\FF_{41})$ is such a point.
  Therefore the local invariant of $\alpha$ at $41$ of the lift of such a point
  is $1/2$, so by \eqref{eqn:p<>41_1} the corresponding adelic point
  cannot be approximated by a rational point.
	
	(2)  The rational point 
	$\bu = (-41/9, -26/3, 4/3)$ satisfies
	$\bu \bmod 41 \equiv (0,5,15) \in \sU_m(\FF_{41})$
	(viewed as an element of $\sU_m(\ZZ_{41})$).
	
	To prove that there is no integral point with this property,
	we show that there is a Brauer-Manin obstruction coming from
	$\alpha_{2,-} = (u_2 - 2, m-4)$. Lemma \ref{lem:invp at most primes} again implies that
	for all $\bu \in \sU_m(\ZZ_p)$ we have
  \begin{equation} \label{eqn:p<>41}
    \inv_p \alpha_{2,-}(\bu) =
      0 \quad \mbox{if } p \neq 41.
  \end{equation}
  At $41$, one has $5 -2 = 3  \notin \FF_{41}^{*2}$,
  hence $\inv_{41} \alpha_{2,-}(\bu) = 1/2$ for any $\bu \in \sU_m(\ZZ_{41})$
  with $\bu \equiv (0,5,15) \bmod 41$. Thus such a point cannot be approximated
  by an integer point, as claimed.
\end{proof}

In fact, as explained in the proof of Corollary \ref{cor:SA},  it follows from \cite{SS91} that the Brauer-Manin obstruction is the only one to weak approximation on $S_m$. So in Proposition \ref{prop:fail_SA}, one knows the stronger claim that any $41$-adic lift of $(0,5,15)$ can be approximated arbitrarily well by a rational point.

\begin{remark}
A family of failures of strong approximation was presented in \cite[\S 8]{GS17}, under the assumption that $n \mapsto \left(\frac{4(m-4)}{n}\right)$ is a primitive Dirichlet character modulo $n$. The example in Proposition \ref{prop:fail_SA} is not covered by this family, as the character $\left(\frac{4\cdot 41}{n}\right)$ is induced from a primitive character modulo $41$, since $41 \equiv 1 \bmod 4$

Moreover, the examples given in  \cite[\S 8]{GS17} involve congruences on the $u_i^2 - 4$. One easily verifies that these come from a Brauer-Manin obstruction associated to $\alpha$, in particular, they are in fact explained by a failure of \emph{weak} approximation.
\end{remark}

\subsection{Proof of Theorem \ref{thm:WA}}
We may assume that $m$ satisfies the hypotheses of Corollary \ref{cor:generic}, as the number of $m$ which fail this are $O(\sqrt{B})$.
If $m-4$ is divisible by a prime $p > 3$ to odd valuation, then $U_m$ fails weak approximation by Corollary \ref{cor:SA}. It is simple to see that the cardinality of $m$ to which this result does not apply is $O(B^{1/2})$. \qed

\subsection{Proof of Theorem \ref{thm:SA}}
Minor variant of the proof of Theorem \ref{thm:WA}. \qed

\subsection{Proof of Theorem \ref{thm:235}}
If $m$ satisfies the conditions of Corollary \ref{cor:generic},
then it follows from Corollary \ref{cor:HP} that any prime $p \mid m-4$
with $p > 5$ must divide $m-4$ to even multiplicity, thus the result holds
in this case.
Assume then that the conditions in Corollary \ref{cor:generic} fail. Then one of the following
holds
$$m-4, -(m-4),m, m(m-4) \in \QQ^{*2}.$$
The first two cases are included in the statement of the theorem.
Moreover, as explained in \S \ref{sec:assumptions}, in the other cases either there is  an integral point or the surface is singular.
This completes the proof. \qed

\subsection{Proof of Theorem \ref{thm:Br asympt}} 
We require the following well-known lemma.

\begin{lemma} \label{lem:number of d's}
	Let $n \in \NN$ and $R \subset (\ZZ/n\ZZ)^*$ be a non-empty 
	set of units.
	Then
	$$\#\{ d \in \NN : d \leq x, p \mid d \implies p \bmod n \in R\}
	\sim c_R x(\log x)^{ |R|/\varphi(n) - 1},$$
	as $x \to \infty$, for some $c_R > 0$.
\end{lemma}
\begin{proof}
	The associated Dirichlet series has the Euler product expansion
	$$F(s) = \prod_{\substack{p \\ p \bmod n \in R}}\left(1 - \frac{1}{p^s}\right)^{-1}.$$
	However for $a \in (\ZZ/n\ZZ)^*$ we have
    $$ \frac{1}{\varphi(n)} \sum_{\chi \bmod n} \chi(a) =
    \begin{cases}
        1,& a \equiv 1 \bmod n, \\
        0,& \text{otherwise}.
    \end{cases}  $$
    It follows that
    \begin{align*}
    	F(s) &= \prod_{\substack{p \\ \gcd(p,n)=1}}\left(1 - \frac{\sum_{r \in R} \sum_{\chi \bmod n}\chi(pr^{-1})}{\varphi(n)p^s}\right)^{-1} \\
    	& = G(s)\zeta(s)^{|R|/\varphi(n)} \prod_{\chi}L(s,\chi)^{(\sum_{r \in R}\chi(r))/\varphi(n)}
    \end{align*}
    where $G$ is holomorphic and non-zero on $\re > 1/2$ and the  product
    is over the non-principal Dirichlet characters modulo $n$.
    The Dirichlet $L$-functions $L(s,\chi)$ are holomorphic and non-zero
    on $\re s \geq 1$, hence $F(s)$ has a holomorphic continuation to the line
    $\re s =1$ away from $s = 1$, where there is a branch point singularity
    of the shape $c'_R/(s - 1)^{|R|/\varphi(n)}$ for some $c'_R \neq 0$. The result
    now follows from a Tauberian theorem (e.g.~\cite[Thm.~II.7.28]{Ten15}).
\end{proof}

For the lower bound, consider $m$ with $m - 4 < 0$ in Proposition~\ref{prop:alpha GS8.1i}. We have $R = \{1,3\}$ and $(\ZZ/8\ZZ)^* = \{1,3,5,7\}$ in Lemma \ref{lem:number of d's}. This gives $N(B) \gg B^{1/2}/(\log B)^{1/2}$ (not $B^{1/2}/(\log B)^{1/4}$ as claimed in \cite[Thm.~1.2.(i)]{GS17}).

Now for the upper bound.
We first take care of those surfaces which do not satisfy the hypotheses of Corollary \ref{cor:generic}. For such surfaces, at least one of the following holds
\[
	m-4,m, m(m-4), -(m-4) \in \QQ^{*2}.
\]
In the first case $\Br U_m = \Br \QQ$, hence there is no Brauer--Manin obstruction. As explained in \S \ref{sec:assumptions}, in the second and third cases either there is an integral point or the surface is singular. In the last case where $-(m-4) \in \QQ^{*2}$, Corollary~\ref{cor:generic} is applicable unless $\frac{\sqrt{m} + \sqrt{m-4}}{2}$ is a square in  $\QQ(\sqrt{m}, \sqrt{m-4})$. By Proposition~\ref{prop:no trans in general} this happens precisely when $m = 4- (2(n^2 \pm 1))^2$, for some $n \in \ZZ_{> 0}$. The collection of such $m$ with $|m| \leq B$ is $O(B^{1/4})$, which is clearly negligible in Theorem \ref{thm:Br asympt}.

Let $m$ now satisfy the hypotheses of Corollary \ref{cor:generic}. By Corollary~\ref{cor:HP}, if there is a Brauer--Manin obstruction to the integral Hasse principle, then $m - 4 = \varepsilon 2^{a_2}3^{a_3}5^{a_5}d^2$ is not a square in $\ZZ$. Here $\varepsilon \in \{\pm 1\}$, the $a_i$ are non-negative integers, $d$ is a positive integer coprime to $30$ and satisfying $p \mid d \implies m - 4 \in \Qp^{*2}$. The cardinality of such $|m| \leq B$ is
\[
	\ll \#\left\{ \varepsilon, a_i,d : 
	\begin{array}{l}
	|4 + \varepsilon 2^{a_2}3^{a_3}5^{a_5}d^2| \le B,\\
	\varepsilon 2^{a_2}3^{a_3}5^{a_5} \notin \QQ^{*2}, 
	(d, 30) = 1,
	p \mid d \implies \varepsilon 2^{a_2}3^{a_3}5^{a_5} \in \QQ_p^{*2}
	\end{array}\right\}.
\] 
For $p \nmid 30$ the condition $\varepsilon 2^{a_2}3^{a_3}5^{a_5} \in \QQ_p^{*2}$ is equivalent to a collection of congruence conditions on $p$ in $(\ZZ/120 \ZZ)^*$, of which there are only finitely many possibilities depending on $\varepsilon$ and on the parities of the $a_i$. Moreover, as $m-4$ is not a square, exactly half of the elements of $(\ZZ/120 \ZZ)^*$ arise.

We consider the case where all  the $a_i$ are odd and $\varepsilon = 1$, the other cases being similar. We let $R \subset (\ZZ/120 \ZZ)^*$ be the corresponding residues. Here the above is
\[
	\ll \#\{ a_i \text{ odd},d :  2^{a_2}3^{a_3}5^{a_5}d^2 \le B, p \mid d 
	\implies p \bmod 120 \in R\}.
\] 
 The contribution from where one of the $i$ satisfies $i^{a_i} > B^{1/4}$ is easily seen to be $O(B^{3/8})$, which is negligible in Theorem \ref{thm:Br asympt}. Hence we may assume that $i^{a_i} \leq B^{1/4}$ for each $i$, so that $B/2^{a_2}3^{a_3}5^{a_5} \to \infty$ as $B \to \infty$. Thus we may apply Lemma \ref{lem:number of d's} for the sum over $d$ and see that the above cardinality is 
 \[
	\ll \sum_{a_i} \frac{(B/2^{a_2}3^{a_3}5^{a_5})^{1/2}}{(\log (B/2^{a_2}3^{a_3}5^{a_5}))^{1/2}}
	\ll \frac{B^{1/2}}{(\log B)^{1/2}}
	\] 
	as the sum over the $a_i$ is convergent. This is sufficient for the upper bound in Theorem \ref{thm:Br asympt}, and completes the proof.
\qed

\subsection{Proof of Theorem \ref{thm:no_Br}}
To prove the theorem, we consider the counter-examples to the Hasse principle given in \cite[Prop.~8.1(ii)]{GS17}.

\begin{proposition}
	Let $\ell \geq 13$ be a prime with $\ell \equiv \pm 4 \bmod 9$ and 
	$\ell \not \equiv \pm 1 \bmod 8$.
	Let $m = 4 + 2\ell^2$. Then $\sU_m(\Adele_\ZZ)^{\Br} \neq \emptyset$ but $\sU_m(\ZZ) = \emptyset$, i.e.~$U_m$ is a counter-example to the integral
	Hasse principle which is not explained by the Brauer--Manin obstruction.
\end{proposition}
\begin{proof}
	That $U_m$ is a counter-example to the integral Hasse principle is shown in
	\cite[Prop.~8.1(ii)]{GS17}. So it suffices to show that there is no Brauer--Manin
	obstruction in this case. This follows from Corollary~\ref{cor:HP} since the assumptions on $\ell$ imply $\ell > 5$ and $m - 4 \notin \QQ_{\ell}^{\ast 2}$.
 \end{proof}

This is clearly sufficient for Theorem  \ref{thm:no_Br}. \qed

\bibliographystyle{amsalpha}{}

\end{document}